\documentclass[11pt]{article}

\usepackage{amsmath,amsthm,amssymb}
\usepackage[usenames,dvipsnames]{xcolor}
\usepackage{enumerate}
\usepackage{graphicx}
\usepackage{cite}
\usepackage{comment}
\usepackage[margin=1in]{geometry}

\usepackage[pdftitle={Closure of LQG},
  pdfauthor={Andres Contreras-Hip and Ewain Gwynne},
colorlinks=true,linkcolor=NavyBlue,urlcolor=RoyalBlue,citecolor=PineGreen,bookmarks=true,bookmarksopen=true,bookmarksopenlevel=2,unicode=true,linktocpage]{hyperref}

\theoremstyle{plain}
\newtheorem{thm}{Theorem}[section]
\newtheorem{theorem}[thm]{Theorem}
\newtheorem{cor}[thm]{Corollary}
\newtheorem{lemma}[thm]{Lemma}
\newtheorem{prop}[thm]{Proposition}

\def\@rst #1 #2other{#1}
\newcommand\MR[1]{\relax\ifhmode\unskip\spacefactor3000 \space\fi
  \MRhref{\expandafter\@rst #1 other}{#1}}
\newcommand{\MRhref}[2]{\href{http://www.ams.org/mathscinet-getitem?mr=#1}{MR#2}}

\theoremstyle{definition}
\newtheorem{defn}[thm]{Definition}

\numberwithin{equation}{section} 

\newcommand{\dsb}{\begin{adjustwidth}{2.5em}{0pt}
\begin{footnotesize}}
\newcommand{\dse}{\end{footnotesize}
\end{adjustwidth}}

\newcommand{\ssb}{\begin{adjustwidth}{2.5em}{0pt}}
\newcommand{\sse}{\end{adjustwidth}}

\newcommand{\aryb}{\begin{eqnarray*}}
\newcommand{\arye}{\end{eqnarray*}}
\def\alb#1\ale{\begin{align*}#1\end{align*}}
\def\allb#1\alle{\begin{align}#1\end{align}}
\newcommand{\eqb}{\begin{equation}}
\newcommand{\eqe}{\end{equation}}
\newcommand{\eqbn}{\begin{equation*}}
\newcommand{\eqen}{\end{equation*}}

\newcommand{\BB}{\mathbb}

\newcommand{\op}{\operatorname}

\newcommand{\ep}{\varepsilon}

\newcommand\p{\partial}
\newcommand\e{\varepsilon}

\newcommand\R{\mathbb{R}}
\newcommand\Z{\mathbb{Z}}
\newcommand\N{\mathbb{N}}

\newcommand\norm[1]{\lVert#1\rVert}

\newcommand\Diam{\mbox{Diam}}
\newcommand\vphi{\varphi}

\let\originalleft\left
\let\originalright\right
\renewcommand{\left}{\mathopen{}\mathclose\bgroup\originalleft}
\renewcommand{\right}{\aftergroup\egroup\originalright}

\title{A support theorem for exponential metrics of log-correlated Gaussian fields in arbitrary dimension}
 \date{ }
 \author{ 
\begin{tabular}{c} Andres A. Contreras Hip and Ewain Gwynne\\ \small University of Chicago \end{tabular}  
}

\begin{document}

\maketitle

\newcommand{\Cupper}{{\hyperref[eqn-bilip-def]{\mathfrak C_*}}}

\begin{abstract}
Let $h$ be a log-correlated Gaussian field on $\R^d$, let $\gamma \in (0,\sqrt{2d}),$ let $\mu_h$ be the $\gamma$-Gaussian multiplicative chaos measure, and let $D_h$ be an exponential metric associated with $h$ satisfying certain natural axioms. In the special case when $d=2$, this corresponds to the Liouville quantum gravity (LQG) measure and metric. We show that the closed support of the law of $(D_h,\mu_h)$ includes all length metrics and probability measures on $\R^d$. That is, if $\mathfrak d$ is any length metric on $\R^d$ and $\mathfrak m$ is any probability measure on $\R^d$, then with positive probability $(D_h , \mu_h)$ is close to $(\mathfrak d , \mathfrak m)$ with respect to the uniform distance and the Prokhorov distance. Key ingredients include a scaling limit theorem for a first passage percolation type model associated with $h$, a special version of the white noise decomposition of $h$ in arbitrary dimension, and an approximation property by conformally flat Riemannian metrics in the uniform sense. 
Our results provide a robust tool to show that the LQG measure and metric, and its higher dimensional analogs, satisfy certain properties with positive probability.  
\end{abstract}

\tableofcontents


\section{Introduction}

Liouville quantum gravity (LQG) is a family of random metric measure spaces parametrized by Riemann surfaces, depending on a parameter $\gamma \in (0,2)$, which describe the scaling limits of random planar maps. 
LQG was first introduced (non-rigourously) by Polyakov in \cite{polyakov-qg1} as a class of canonical models of random surfaces. 
One can define LQG surfaces with the topology of any desired orientable surface~\cite{wedges,dkrv-lqg-sphere,drv-torus,remy-annulus,grv-higher-genus}. However, in this paper we will only consider LQG surfaces with the topology of the whole plane.

The main results of this paper describe the closed support of the law of the LQG metric measure space. More precisely, we show that if $\mathfrak d$ is any length metric on $\R^2$ and $\mathfrak m$ is any probability measure on $\R^2$, then with positive probability $(D_h , \mu_h)$ is close to $(\mathfrak d , \mathfrak m)$ with respect to the uniform distance and the Prokhorov distance (Theorems~\ref{flat} and Corollary~\ref{measandmet}).
Our results provide a convenient ``black box" for whenever one needs to show that certain events for LQG hold with positive probability.

Several recent works (see, e.g.,~\cite{dgz-exponential-metric,dhks-discrete-to-cont,dhks-even-dim,cercle-higher-dimension}) have studied analogs of LQG in dimension $d\geq 3$. The results of this paper also extend to this higher-dimensional setting, conditional on some properties of the higher-dimensional analog of the LQG metric whose proofs have not yet been written down.  In this setting, our results raise important questions about what it means for a random Riemannian metric to be ``conformally flat"; see Section~\ref{sec:higher-dim} for more details. 

Roughly speaking, the aforementioned works on higher-dimensional analogs of LQG consider a conformally flat Riemannian metric tensor on $\mathbb R^d$, $d\geq 3$, of the form
\begin{equation}  \label{eqn:higher-dim-metric}
e^{\gamma h} (dx_1^2 + \dots + dx_d^2)  ,
\end{equation}
where $\gamma \in (0,\sqrt{2d})$ and $h$ is a log-correlated Gaussian field on $\mathbb R^d$, or some variant thereof (the works~\cite{dhks-even-dim,cercle-higher-dimension} also consider similar Riemannian metric tensors on $d$-manifolds for even values of $d$). 

As in the two-dimensional case, the Riemannian metric tensor~\eqref{eqn:higher-dim-metric} does not make literal sense since $h$ is a generalized function and is not pointwise defined. However, one can construct the volume form associated with~\eqref{eqn:higher-dim-metric} as a random measure on $\mathbb R^d$ via the general theory of Gaussian multiplicative chaos~\cite{kahane,rhodes-vargas-review,berestycki-gmt-elementary,bp-lqg-notes}. 

The recent paper~\cite{dgz-exponential-metric} proves the tightness of a natural approximation scheme for the Riemannian distance function associated with~\eqref{eqn:higher-dim-metric}. This paper also shows that every subsequential limit is a metric on $\mathbb R^d$ which induces the Euclidean topology. It is expected, but not yet proven, that the subsequential limit is unique and is characterized by a similar list of axioms as in the two-dimensional case~\cite{gm-uniqueness}; see Definition \ref{LQGlike}. Our results for $d\geq 3$ are proven conditional on these axioms.

\subsection{Measure and metric definitions}
\label{sec:lqg-def}
 
We will now discuss the definitions of the measure and metric associated with the log-correlated Gaussian field $h.$

\begin{defn} \label{def:gff}
The \textbf{log-correlated Gaussian field} is the centered Gaussian random generalized function $h^0$ defined on $\R^d$ up to additive constants, with covariance structure
\begin{equation} \label{eqn:cov}
\op{Cov}\left( h^0(z) , h^0(w) \right) = G(z,w) := \log\left(\frac{1}{|z-w|} \right) .
\end{equation} 
That is, if $f$ is a smooth compactly supported test function on ${\R^d}$ with $\int_{\p B_1(0)}f d\sigma =0,$ then $\int_{{\R^d}} f(z) h^0(z)\,dz$ is a centered Gaussian random variable with variance $\int_{\R^d} \int_{\R^d} f(z) f(w) \log\left(\frac{1}{|z-w|} \right)\,dz\,dw$.
We define the field $h$ to be the field $h^0$ normalized so that is unit sphere average is zero, i.e.,
\[
h = h^0 -  \int_{\p B_1(0)} h^0(z) d\sigma(z) .
\]
In dimension $d=2,$ the field $h$ is the whole-plane Gaussian free field (appropriately normalized) and has the explicit covariance structure given by
\begin{equation}
\op{Cov}\left( h(z) , h(w) \right)   := \log\left(\frac{\max\{|z|,1\} \max\{|w|,1\}}{|z-w|} \right) .
\end{equation} 
\end{defn}
For an exposition of the Gaussian free field in $d=2$, see, e.g., \cite{pw-gff-notes,bp-lqg-notes,shef-gff}. For an exposition of log-correlated Gaussian fields in general dimension, see, e.g., \cite{fgf-survey,lgf-survey}.

Formally, for $\gamma \in (0,\sqrt{2d})$, the \textbf{$\gamma$-exponential geometry} associated with $h$ is described by the random Riemannian metric tensor
\begin{equation} \label{eqn:lqg-tensor}
e^{\gamma h} \, (dx_1^2+ \cdots + dx_d^2) , 
\end{equation}
where $dx_1^2+ \cdots + dx_d^2$ denotes the Euclidean metric tensor on ${\R^d}$. This metric tensor does not make literal sense since $h$ is not a random function but instead it is a random generalized function, and hence $e^{\gamma h}$ is not well defined. However, at least in dimension $2,$ we can still define a random metric and measure associated with~\eqref{eqn:lqg-tensor}. To do so, we first approximate $\tilde h$ with well chosen mollifications, then pass to the limit.

One possible construction in dimension $d=2$ is as follows. Let $p_t$ denote the heat kernel in dimension $2,$
\begin{equation}\label{heatkernel}
p_t(z):=\frac{1}{2\pi t}e^{-\frac{\vert z\vert^2}{2t}}.
\end{equation}
Define the mollified log-correlated Gaussian field
\[
h_\e(z) := p_{\e^2/2} \ast h(z) = \int_{\R^d} h(w) p_t(z-w) \,dw ,
\]
where the integral is interpreted in the distributional sense.
We can then define the \textbf{$\gamma$-LQG measure} as the a.s.\ weak limit
\begin{equation} \label{eqn:lqg-measure}
\lim_{\e\to 0} \e^{\gamma^2/2} e^{\gamma h_\e(z)} .
\end{equation}
This construction is a special case of the theory of \textbf{Gaussian multiplicative chaos}~\cite{kahane,rhodes-vargas-review,berestycki-gmt-elementary,bp-lqg-notes}. The measure $\mu_h$ is locally finite, non-atomic, assigns positive mass to every open set, and is mutually singular with respect to Lebesgue measure. This construction is also possible in higher dimensions. We also have analogous properties for the Gaussian multiplicative chaos measure $\mu_h:$
\begin{itemize}
\item[1.] Almost surely, $\mu_h$ is a non-atomic Radon measure.
\item[2.] If $U\subset {\R^d}$ is any deterministic open set, then $\mu_h(U)$ is given by a measurable function of $h\vert_U.$
\item[3.] Almost surely, we have that for any continuous function $f:{\R^d}\to \R,$ $e^{\gamma f}\cdot \mu_h = \mu_{h+f}.$
\item[4.] Almost surely, we have the following statement. Suppose that $\phi$ is a composition of dilations, translations, and rotations. Then
\[
\mu_{h\circ \phi + Q \log \vert \phi'\vert}(A) = \mu_h(\phi(A))
\]
for any measurable $A \subset \R^d.$
\end{itemize}

The $\gamma$-LQG metric is defined similarly. 
Let $d_\gamma > 2$ be the fractal dimension of $\gamma$-LQG. 
Prior to the construction of the LQG metric, this number was shown to arise in various approximations of the LQG metric in~\cite{dzz-heat-kernel,dg-lqg-dim}. 
After the metric was constructed, it was shown that $d_\gamma$ is its Hausdorff dimension~\cite[Corollary 1.7]{gp-kpz}. 
We note that $d_\gamma$ is not known explicitly except that $d_{\sqrt{8/3}}=4$. 
Let
\begin{equation} \label{eqn:xi}
\xi := \frac{\gamma}{d_\gamma} .
\end{equation}

Similarly to~\eqref{eqn:lqg-measure}, we define the approximating metrics
\begin{equation} \label{eqn:lfpp}
D_{h}^\e(z,w) := \inf_{P : z\to w} \int_0^1 e^{\xi h_\e(P(t))} |P'(t)| \,dt
\end{equation}
where the infimum is over all piecewise continuously differentiable paths from $z$ to $w$. It was shown in the series of papers~\cite{dddf-lfpp,local-metrics,lqg-metric-estimates,gm-confluence,gm-uniqueness} that there are normalizing constants $\{\mathfrak a_\e\}_{\e > 0}$ and a random metric $D_{h}$ on ${\R^d}$ such that 
\begin{equation} \label{eqn:lqg-metric}
D_h(z,w) = \lim_{\e \to 0} \mathfrak a_\e^{-1} D_{h}^\ep(z,w) 
\end{equation}
in probability with respect to the topology of uniform convergence on compact subsets of ${\R^d}\times{\R^d}$. The metric $D_h$ is defined to be the \textbf{$\gamma$-LQG metric}. The metric $D_h$ induces the same topology on ${\R^d}$ as the Euclidean metric, but has very different geometric properties. See~\cite{ddg-metric-survey} for a survey of known results about the LQG metric. 

In higher dimensions, the metric measure space $D_h$ associated with the log-correlated Gaussian field $h$ can be defined, but is not yet defined uniquely \cite{dgz-exponential-metric}. However, it is conjectured that such a metric can be defined uniquely, and that it satisfies the following ``strong" axioms. First, we define length metrics.
\begin{defn} \label{def-length}
Let $(X,\mathfrak d)$ be a metric space. 
For a path $P : [a,b]\to X$, we define the \emph{$\mathfrak d$-length} of $P$ by
\eqbn
\ell_{\mathfrak{d}}(P) := \sup_{T} \sum_{i=1}^{\# T} \mathfrak d(P(t_i) , P(t_{i-1})) 
\eqen
where the supremum is over all partitions $T : a= t_0 < \dots < t_{\# T} = b$ of $[a,b]$.  
We say that $\mathfrak d$ is a \textbf{length metric} if for each $x,y\in X$, the distance $\mathfrak d(x,y)$ is equal to the infimum of the $\mathfrak d$-lengths of the $\mathfrak d$-continuous paths from $x$ to $y$.
\end{defn}
Now we will define the Weyl scaling as follows: for any continuous $f:{\R^d}\to\R$ and any length metric $\mathfrak{d}$ on $\mathbb R^d$ inducing the Euclidean topology, we let
\begin{equation}\label{weyl}
(e^f \cdot \mathfrak{d})(x,y) = \inf_P \int_0^{\ell(P)} e^{f(P(t))} dt 
\end{equation}
where $\ell(P)$ is the $\mathfrak{d}$-length of $P$ and the inf is over all paths from $x$ to $y$ parametrized by $\mathfrak{d}$-length. For any open set $U\subseteq {\R^d}$ we also define the internal metric $\mathfrak{d}(\cdot,\cdot ; U)$ by
\[
\mathfrak{d}(x,y ; U) = \inf_{P \subseteq U} \ell_{\mathfrak{d}}(P),
\]
where the infimum is taken over paths $P:[0,1]\to U$ such that $P(0)=x$ and $P(1)=y.$
Then we will assume the metric $D_h$ satisfies the following axioms.

\begin{defn}\label{LQGlike}
We say a map from random generalized random functions $h$ on $\R^d$ to metric measure spaces $(\R^d,D_h,\mu_h)$ satisfies the strong LQG axioms if there exist constants $\xi,\gamma$ such that the following hold:
\begin{itemize}
\item[1.] Almost surely, $D_h$ is a length metric.
\item[2.] Let $U$ be a deterministic open set. Then the internal metric $D_h(\cdot ,\cdot ; U)$ is a measurable function of $h \vert_U.$
\item[3.] Weyl scaling: Recall the definition of Weyl scaling \eqref{weyl}.
Then, a.s., for every continuous function $f$, we have $e^{\xi f} \cdot D_h = D_{h+f}$. 
\item[4.] Coordinate change: let $\lambda>0$ and $z_0\in {\R^d}.$ Then almost surely,
\[
D_h(\lambda x+z_0,\lambda y+z_0) = D_{h(\lambda \cdot + z_0)+Q\log \lambda}(x,y) ,\quad \forall x,y\in \R^d
\]
where $Q:=\frac{\gamma}{2}+\frac{d}{\gamma}.$
\end{itemize}
\end{defn}
Note that the above axioms are satisfied by the LQG metric in dimension $d=2$ (see \cite{gm-uniqueness} and \cite{lqg-metric-estimates}). It is believed that the above axioms also hold for any subsequential limit obtain in \cite{dgz-exponential-metric} for $d \geq 3,$ but no proof of this has been written yet.

We will also assume that
\begin{equation}\label{gammaxiass}
\frac{\gamma}{\xi} >d
\end{equation}
as in the two dimensional case. We expect this to hold, since one expects $\frac{\gamma}{\xi}= \dim(\R^d;D_h)$ and $\dim(\R^d;D_h)>d$ since $D_h$'s geometry should be strictly ``rougher" than Euclidean. We expect this can be proven rigorously with comparison arguments, but we do not do this here, and is the subject of forthcoming work of the first author and Z. Zhuang. 

We will additionally assume that
\begin{equation}\label{extramombound}
\mathbb{E}(\mathrm{diam}_{D_h}([0,1]^d)^{\bar{p}}) < \infty
\end{equation}
for some $\bar{p}>d.$ This also holds for the LQG metric in dimension $2$ (see \cite[Proposition 3.9]{lqg-metric-estimates}). The axioms above are all used throughout the paper, while the bounds \eqref{gammaxiass} and \eqref{extramombound} are technical assumption, only used in Subsection~\ref{sec-particular} and Subsection \ref{fppargument}, respectively.

\subsection{Main results}

Recall that our goal is to determine the closed support of the law of $({\R^d} , D_h,\mu_h)$ with respect to some reasonable topology on the space of metric measure spaces, assuming $h$ is as in Definition \ref{def:gff}.

The LQG metric (and any higher dimensional analogue) is a length metric essentially by construction. The class of length metrics is preserved under various forms of convergence, e.g., uniform convergence~\cite[Exercise 2.4.19]{bbi-metric-geometry(dupe)} and Gromov-Hausdorff convergence~\cite[Theorem 7.5.1]{bbi-metric-geometry(dupe)}. Hence, the LQG metric cannot approximate a metric which is not a length metric. Our first main result implies that this is essentially the only constraint.

We say that $(X,\mathfrak d)$ is \textbf{boundedly compact} is each closed bounded subset of $X$ is $\mathfrak d$-compact.\footnote{The reason for imposing the bounded compactness assumption is to make it so that all of the $\mathfrak d$-geodesics going between points in $[-R,R]^d$ are contained in some ($\mathfrak d$-dependent) compact set. We expect that Theorem~\ref{flat} is true without the boundedly compact assumption, and that this can be obtained by approximating the restriction to $[-R,R]^d$ of a length metric which is not boundedly compact by the corresponding restrictions of boundedly compact length metrics. However, for the sake of brevity, we do not carry this out here.}

\begin{theorem} \label{flat} 
Let $\mathfrak{d}$ be a boundedly compact length metric on $\R^d$ which induces the Euclidean topology and let $\mathfrak{m}$ be a locally finite Borel measure on $\R^d$. Suppose that there exists a map $h \mapsto (\R^d,D_h,\mu_h)$ as in Definition \ref{LQGlike} with $\xi < 1,$ and suppose $h$ is a log-correlated Gaussian field. Let $D_h,\mu_h$ be as Definition \ref{LQGlike}. Then $D_{h}$ and $\mu_{h}$ approximate the metric $\mathfrak{d}$ and the measure $\mathfrak{m}$ respectively with positive probability. More precisely, for any $\e,R>0,$ we have with positive probability that
\[
\left\vert D_{h}(x,y)-\mathfrak{d}(x,y) \right\vert<\e
\]
for all $x,y\in [-R,R]^d$ and at the same time for any Borel set $A\subseteq[-R,R]^d$ we have
\[
\mathfrak{m}(A) \leq \mu_{h}(A_\e)+\e
\]
and
\[
\mu_{h}(A) \leq \mathfrak{m}(A_\e)+\e
\]
where $A_\e:= \{ z\in {\R^d}: \inf_{u\in A} \vert z-u\vert < \e\}.$
\end{theorem}


Perhaps surprisingly, we do not need to assume any relationship between $\mathfrak d$ and $\mathfrak m$ in Theorem~\ref{flat}. 
Even though $D_h$ and $\mu_h$ are closely related to each other (see, e.g., \cite{gs-lqg-minkowski}), there is a positive chance for them to behave quite differently. 
A much weaker result in this direction in the case of $d=2$ appears as~\cite[Proposition 11.9]{bg-harmonic-ball}.

In dimension two, Theorem \ref{flat} can be viewed as a general-purpose theorem for making the LQG measure and metric do things with positive probability. 
In this sense, these results are an LQG analog of ``support theorems" for various stochastic processes, which characterize the elements of the state space that they can approximate with positive probability. Examples of such support theorems include the classical fact that $d$-dimensional Brownian motion can be made to approximate any continuous path in $\mathbb R^d$; and analogous statements for Schramm-Loewner evolution (see, e.g.,~\cite[Section 2]{miller-wu-dim}). 

We expect that Theorem \ref{flat} and the some of the techniques used in its proof will be useful in future works concerning LQG and its higher dimensional analogs. 
Indeed, when studying LQG, one often needs to show that $D_{h}$ and/or $\mu_{h}$ has some prescribed behavior with positive probability. Prior to this work, this was typically done via ad hoc methods based on the fact that adding a smooth bump function to $h$ changes its law in an absolutely continuous way. See~\cite[Section 4.1]{ddg-metric-survey} for an explanation of this technique. Arguments of this type can sometimes be quite complicated, see, e.g.,~\cite[Sections 11.2 and 11.3]{harmonicballs} or \cite[Section 5]{uniquenessoflqg}. In future work, such arguments could be replaced by applications of Theorem~\ref{flat}.

In the course of proving Theorem \ref{flat} in dimension $d\geq 3$, we also establish several independently interesting facts about log-correlated Gaussian fields which we expect to be useful elsewhere (see e.g. Lemmas \ref{samelaw}, \ref{avggoingto0}, and \ref{CMhigherdim}). 

\subsection{Interpretation of results in dimension \texorpdfstring{$d\geq 3$}{d3}}
\label{sec:higher-dim}
 
In dimension $d\geq 3$, Theorem~\ref{flat} is of interest for the same reasons as in the two-dimensional case, as discussed just above. It also raises interesting questions about the interpretation of the random Riemannian metric tensor~\eqref{eqn:higher-dim-metric}, which we will discuss in this section. 

A Riemannian metric tensor $g$ on a domain $U\subset \mathbb R^d$ is \textbf{conformally flat} if there is a continuous function $f : U \to \mathbb R$ such that $g = e^f (dx_1^2 + \dots + dx_d^2)$, where $dx_1^2 + \dots + dx_d^2$ is the Euclidean Riemmanian metric tensor. More generally, a $d$-dimensional Riemannian manifold is \textbf{locally conformally flat} if it admits an atlas where in each chart, the Riemannian metric is conformally flat.  
 It is a classical theorem that every two dimensional Riemannian manifold is locally conformally flat, i.e., it can locally be expressed in isothermal coordinates. However, this is not true in dimension $d \geq 3$. Rather, a $d$-dimensional Riemannian manifold is locally conformally flat if and only if its Cotton tensor (for $d=3$) or its Weyl tensor (for $d\geq 4$) is identically zero. 
 
Theorem~\ref{flat} implies that the random metric and measure associated with~\eqref{eqn:higher-dim-metric} can approximate \emph{any} length metric and locally finite Borel measure on $\mathbb R^d$ with positive probability, even a metric and measure associated with a Riemannian metric tensor which is not conformally flat. This leads one to wonder whether the random Riemannian metric tensor associated with~\eqref{eqn:higher-dim-metric} should be thought of as a random conformally flat metric on $\mathbb R^d$, or as a general random Riemannian metric on $\mathbb R^d$. This is closely related to some fundamental questions about analogs of LQG in higher dimensions, e.g., the following. 
\begin{itemize}
\item Let $n \in \mathbb N$ and let $M_n$ be a uniform sample from the set of triangulations of the $d$-sphere (i.e., simplicial complexes homeomorphic to $\mathbb S^d$) with $n$ $d$-simplices. For $d=2$, it is known that $M_n$ (equipped with its graph distances, appropriately rescaled) converges in the Gromov-Hausdorff sense to $\sqrt{8/3}$-LQG~\cite{legall-uniqueness,miermont-brownian-map,lqg-tbm1,lqg-tbm2} (see also~\cite{hs-cardy-embedding} for a stronger topology of convergence).  For $d\geq 3$, does $M_n$ converge in the scaling limit (e.g., with respect to the Gromov-Hausdorff distance) to a $d$-manifold, equipped with a random Riemannian distance function of the form~\eqref{eqn:higher-dim-metric}? Or is it the case that to get convergence to a random Riemannian distance function of the form~\eqref{eqn:higher-dim-metric}, one needs to instead take a uniform sample from a restricted set of triangulations of $\mathbb S^d$ which are required to satisfy some discrete analog of conformal flatness? See~\cite{dgz-exponential-metric,bg-rw-sphere-packing} for related discussions.
\item Consider the following random Riemannian metric tensor on $\mathbb R^d$: 
\begin{equation} \label{eqn:higher-dim-general}
\sum_{i,j=1}^d e^{\beta h_{i,j}} dx_i dx_j 
\end{equation}  
where the $\beta > 0$ is a parameter and $h_{i,j}$s are log-correlated Gaussian fields on $\mathbb R^d$ (coupled together in some manner). 
Can one make rigorous sense of the distance function associated with~\eqref{eqn:higher-dim-general}? If so, then this distance function superficially seems to be supported on a more general class of metrics than~\eqref{eqn:higher-dim-metric} since~\eqref{eqn:higher-dim-general} is not required to be conformally flat. However, our results show that any length metric on $\mathbb R^d$ can be approximated with positive probability by the Riemannian distance function associated with~\eqref{eqn:higher-dim-general} (conditional on the conjecture that this Riemannian distance function satisfies the strong axioms). So, in what sense does~\eqref{eqn:higher-dim-general} actually give a more general Riemannian metric than~\eqref{eqn:higher-dim-metric}? Could it be that the limit of a natural approximation scheme for the Riemannian distance function associated with~\eqref{eqn:higher-dim-general} is actually just the Riemannian distance function associated with~\eqref{eqn:higher-dim-metric}, where $h$ is some function of $\{h_{i,j}:i,j=1,\dots,d\}$ and $\gamma$ is some function of $\beta$ and $d$?
\end{itemize}
We do not know the answers to the above questions, even conjecturally, but they seem to be worth pursuing further.

\subsection{Applications of Theorem \ref{flat} in the two-dimensional case}
\subsubsection{Application to the LQG sphere}
As a consequence of Theorem \ref{flat}, one can obtain a similar result concerning a special LQG surface called the $\gamma$-LQG sphere, which is represented by a certain special variant of the GFF. 
There are a number of equivalent ways of defining the $\gamma$-LQG sphere. The first definitions appeared in~\cite{wedges,dkrv-lqg-sphere} and were proven to be equivalent in~\cite{ahs-sphere}. The definition we give here is~\cite[Definition 2.2]{ahs-sphere} with $k=3$ and $\alpha_1=\alpha_2=\alpha_3=\gamma$.

\begin{defn} \label{def:sphere}
Let $\gamma\in (0,2)$ and let
\begin{equation} \label{eqn:Q} 
Q := \frac{2}{\gamma} + \frac{\gamma}{2} .
\end{equation} 
The \textbf{Liouville field} is the random generalized function 
\begin{equation} \label{eqn:hL}
h = h^{\mathbb{C}}+\gamma G(0,\cdot)+\gamma G(1,\cdot)-(2Q-\gamma)\log\vert \cdot \vert_+ 
\end{equation}
where $h^\mathbb{C}$ is the whole-plane GFF,  $G(\cdot,\cdot)$ is its covariance kernel as in~\eqref{eqn:cov}, and $\log\vert \cdot\vert_+:= \log (\max\{\vert z\vert,1\}).$ The \textbf{(triply marked, unit area) $\gamma$-LQG sphere} is the LQG surface parametrized by $\mathbb{C}$ represented by the field $h$, where the law of $\BB P_h$ of $h$ is equal to the law $\BB P_{\tilde h}$ of
\begin{equation} \label{eqn:tilde-hL}
\tilde h := h - \frac{1}{\gamma} \log \mu_{h}(\mathbb{C}) 
\end{equation}
weighted by a $\gamma$-dependent constant times $[\mu_{h}(\mathbb{C})]^{4/\gamma^2-2}$, i.e., $\frac{d\BB P_h}{d\BB P_{\tilde h}} = C(\gamma) \times [\mu_{h}(\mathbb{C})]^{4/\gamma^2-2}$.

\end{defn}

We note that if $h$ is as in Definition~\ref{def:sphere}, then $\mu_h(\mathbb{C}) = 1$ (this is because of the subtraction of $\frac{1}{\gamma} \log \mu_{h}(\mathbb{C}) $ in~\eqref{eqn:tilde-hL}). 
Furthermore, one can check that the LQG metric associated with the LQG sphere extends continuously to the one-point compactification $\mathbb{C} \cup \{\infty\}$, so it can be viewed as a metric on the sphere (not just on $\mathbb{C}$). 
One reason why the $\gamma$-LQG sphere is special is that it is the LQG surface which arises as the scaling limit of random planar maps with the sphere topology. See e.g. subsubsection~\ref{sec:rpm} for more details.

We also have the following result for LQG in dimension 2.
\begin{cor}\label{measandmet} 
Let $\mathfrak d$ be a length metric on the sphere $\mathbb S^2 = \BB C \cup \{\infty\}$ which induces the Euclidean topology and let $\mathfrak m$ be a Borel measure on $\mathbb S^2$ with $\mathfrak m(\mathbb S^2 )  =1$. 
Fix $\gamma\in(0,2)$ and let $D_{\tilde{h}}$ and $\mu_{\tilde{h}}$ denote the metric and measure associated with the $\gamma$-LQG quantum sphere, viewed as a metric and measure on $\mathbb S^2$. 
For each $\e  > 0$, it holds with positive probability that
\begin{equation}\label{metriccomparisons}
\left\vert D_{\tilde{h}}(x,y) -\mathfrak{d}(x,y) \right\vert \leq\e, \quad \forall x,y\in \mathbb S^2 
\end{equation}
and at the same time for any Borel set $A \subset \mathbb S^2$,
\begin{equation}\label{measurecomp1}
\mathfrak{m}(A) \leq \mu_{\tilde{h}}(A_\e)+\e
\end{equation}
and
\begin{equation}\label{measurecomp2}
\mu_{h}(A) \leq \mathfrak{m}(A_\e)+\e
\end{equation}
where $A_\e:= \{ z\in \mathbb{C}: \inf_{u\in A} d_S(z,u) < \e\},$ with $d_S$ being the Euclidean metric on $\mathbb{S}^2.$
In particular, with positive probability we have that 
\begin{equation}
 d_{GHP} \left( (\mathbb S^2 , \mathfrak d , \mathfrak m) , (\mathbb S^2 , D_{\tilde{h}} , \mu_{\tilde{h}})      \right) \leq \e .
\end{equation} 
\end{cor}
To obtain Corollary \ref{measandmet}, one can write the quantum sphere field ${\tilde{h}}$ in terms of the Gaussian free field,
\[
\tilde{h} := \bar{h} - \frac{1}{\gamma} \log \mu_{\bar{h}}(\mathbb{C}),
\]
where
\[
\bar{h} = h+\gamma G(0,\cdot)+\gamma G(1,\cdot)-(2Q-\gamma)\log\vert \cdot \vert_+.
\]
Since Theorem \ref{flat} lets us ``fix" the total measure $\mu_h,$ it suffices to consider $\bar{h},$ and show that $(D_{\bar{h}},\mu_{\bar{h}})$ is close to any conformally flat Riemannian metric $e^f \cdot d_0$ with positive probability. For this, we first approximate the function $\gamma G(0,\cdot)+\gamma G(1,\cdot)-(2Q-\gamma)\log\vert \cdot \vert_+$ by a smooth, finite Dirichlet energy function $f_\e.$ This approximation can be chosen so that $f_\e(z) = \gamma G(0,z)+\gamma G(1,z)-(2Q-\gamma)\log\vert z \vert_+$ for $z \in B_{\frac{1}{\e}}(0)\setminus (B_\e(0)\cup B_\e(1)).$ By the Cameron Martin property of the GFF, we know that Corollary \ref{measandmet} holds if we replace $\tilde{h}$ with $h+f_\e.$ Now the final step is to show that both the measure and metric are ``unaffected" by the singularities at $0,1, \infty.$ To do this, we note that if $P$ is any geodesic between two points, then we can replace $P$ by a path with similar LQG length avoiding $\e$ neighborhoods of the singularities. A more detailed version of this argument can be found in the previous Arxiv version of this paper.

\subsubsection{Application to random planar maps}
\label{sec:rpm}

Due to the convergence of uniform random planar maps toward $\sqrt{8/3}$-LQG, Theorem~\ref{measandmet} has implications for the study of random planar maps. 
To explain this, let $M_n$ be sampled uniformly from the set of all quadrangulations of the sphere (planar maps whose faces all have degree 4) with $n \in \BB N$ total faces. 
Let $D_n$ and $\mu_n$ denote the graph distance and the counting measure on vertices of $M_n$, respectively. We will now recall the Gromov-Hausdorff-Prokhorov distance. Suppose that $\mathfrak m_1,\mathfrak m_2$ are Borel measures on a metric space $(X,\mathfrak{d}).$ The Prokhorov distance is given by
\[
d_{\mathrm{P}} (\mathfrak m_1,\mathfrak m_2) :=\inf\{\e>0:\mathfrak m_1(A)\leq \mathfrak m_2(A_\e)+\e,\; \mathfrak m_2(A)\leq \mathfrak m_1(A_\e)+\e \mbox{ for all closed }A\},
\]
where
\[
A_\e:=\{x\in X: \inf_{y\in A}\mathfrak{d}(x,y)<\e\}.
\]
Also, recalling that the Hausdorff distance between pairs of sets $A,B$ in a metric space $(X,\mathfrak{d})$ is defined by
\[
d_H(A,B):= \inf\{r>0:A \subseteq \cup_{x\in B}B_{\mathfrak{d}}(x),\;B \subseteq \cup_{x\in A}B_{\mathfrak{d}}(x)\}.
\]
We define the Gromov-Hausdorff distance between two metric spaces $(X_1,\mathfrak{d}_1)$ and $(X_2,\mathfrak{d}_2)$ as the infimum over all $r>0$ such that there exists a third metric space $(Z,\mathfrak{d})$ and subspaces $X_1', X_2'$ that are isometric to $X_1$ and $X_2$ respectively, and additionally $d_H(X_1',X_2')<r.$ Finally, the Gromov-Hausdorff-Prokhorov distance between two metric measure spaces $(X_1,\mathfrak{d}_1,\mu_1)$ and $(X_2,\mathfrak{d}_2,\mu_2)$ is defined as
\[
d_{GHP}((X_1,\mathfrak{d}_1,\mu_1),(X_2,\mathfrak{d}_2,\mu_2))=\inf_{(Z,\mathfrak{d}),\iota_1,\iota_2}d_H(X_1,X_2)+d_{P}(\iota_1^*\mu_1, \iota_2^*\mu_2),
\]
where the infimum is taken over isometric embeddings $\iota_1:X_1 \to Z$ and $\iota_2:X_2 \to Z.$

It was shown independently by Le Gall~\cite{legall-uniqueness} and Miermont~\cite{miermont-brownian-map}, building on many other works, that the metric measure spaces\footnote{Often this scaling limit result is stated with an additional constant factor in front of $n^{-1/4} D_n$. For convenience we implicitly re-scale the metric on the Brownian map so that this constant is not needed. We do the same for the $\sqrt{8/3}$-LQG metric. This re-scaling does not affect the statement of Corollary~\ref{rpm}.}
 $(M_n , n^{-1/4} D_n , n^{-1} \mu_n)$ converge in law in the Gromov-Hausdorff-Prokhorov sense to a random metric measure space called the \textbf{Brownian map}. See~\cite{legall-sphere-survey} for a survey of this work and~\cite{ab-simple,bjm-uniform,aa-odd-angulation,marzouk-degree-sequence} for extensions to other types of random planar maps with the sphere topology. 

Subsequently, Miller and Sheffield~\cite{brownianliouville1,brownianliouville2,brownianliouville3} constructed a metric associated with $\sqrt{8/3}$-LQG (using a very different construction from the one described above). They then showed that the $\sqrt{8/3}$-LQG sphere, equipped with this metric and its LQG area measure, is isometric to the Brownian map. 
Finally, it was shown in~\cite[Corollary 1.4]{gm-uniqueness} that the Miller-Sheffield metric a.s.\ coincides with the metric from~\eqref{eqn:lqg-metric} for $\gamma=\sqrt{8/3}$ ($\xi = 1/\sqrt 6$). Combining these results shows that if $h$ is the random generalized function associated with the $\sqrt{8/3}$-LQG sphere and $M_n$ is a uniform quadrangulation as above, then 
\begin{equation} \label{eqn:rpm-ghp}
(M_n , n^{-1/4} D_n , n^{-1} \mu_n) \to \left( \BB S^2 , D_h , \mu_h \right)  
\end{equation}
in law with respect to the Gromov-Hausdorff-Prokhorov distance. The following corollary is immediate from Theorem~\ref{measandmet} and~\eqref{eqn:rpm-ghp}.

\begin{cor} \label{rpm}
For $n\in\mathbb N$, let $M_n$ be a uniform quadrangulation of the sphere with $n$ faces, as above. 
Let $(X, \mathfrak d , \mathfrak m)$ be a metric space homeormorphic to the sphere, equipped with a length metric and a probability measure.
For each $\epsilon  > 0$, there exists $p = p(\epsilon , (X,\mathfrak d , \mathfrak m)) > 0$ such that for each sufficiently large $n\in\mathbb N$, 
\begin{equation}
\mathbb P\left[ d_{\mathrm{GHP}}\left( (M_n , n^{-1/4} D_n , n^{-1} \mu_n) ,  (X,\mathfrak d , \mathfrak m) \right) < \epsilon \right] \geq p .
\end{equation}
The same holds for other classes of random planar maps known to converge to the Brownian map, e.g., uniform $k$-angulations for $k\geq 3$~\cite{legall-uniqueness,aa-odd-angulation} and uniform planar maps with unconstrained face degree~\cite{bjm-uniform}.
\end{cor}


\subsection{Outline}
  
In Section \ref{flatproof} we prove the following variant of Theorem \ref{flat}:
\begin{theorem}\label{particular}
Let $\e>0,$ and let $f : \R^d\to \R$ be a smooth compactly suported function. Then with positive probability we have that
\[
\sup_{x,y\in [-R,R]^d}\vert D_{h}(x,y)-e^f \cdot d_0(x,y)\vert \leq \e
\]
and at the same time,
\[
\mu_{h}([-R,R]^d) \leq \e.
\]
\end{theorem}
To prove this, we use an analogue of the white noise decomposition for the Gaussian free field to decompose $h$ as $h=h_{0,\beta}+h_{\beta,\infty}$, where $\beta>0$ is small, $h_{0,\beta}$ is a random generalized function, and $h_{\beta,\infty}$ is an independent random smooth function (see~\eqref{eqn:wn-decomp} for a precise statement). The generalized function $h_{0,\beta}$ has very weak long-range correlations and its law is invariant under rotations and translations of ${\R^d}$. The function $h_{\beta,\infty}$ does not enjoy these properties, but it is smooth. Once we do this, we discretize $\R^d$ and use a first passage percolation argument as in \cite{indepperc} followed by scaling arguments for LQG to show that $\beta^{-(\xi Q-1)} D_{h_{0,\beta}}$ is approximately Euclidean and simultaneously $\beta^{-(\gamma Q -d)}\mu_{h_{0,\beta}}$ is approximately the Lebesgue measure with high probability when $\beta$ is small (Propositions \ref{approxeucdist} and \ref{approxeucmeas}). After this, in Lemma \ref{consthinfty} one controls the smooth part corresponding to $h_{\beta,\infty}$ by forcing it to be close to any given function with positive probability. Forcing $h_{\beta,\infty}$ to be approximately $\frac{\xi Q-1}{\xi} \log\beta$ and using Weyl scaling, one obtains that with positive probability, that $D_h$ approximates any conformally flat metric, while $\mu_h([-R,R]^d)$ is bounded by a multiple of $\theta =O(\beta^{\frac{\gamma}{\xi} -d}).$


In Section \ref{pfofsphere} we show that the measure can be modified in such a way that we do not change the metric very much. In Proposition \ref{adddeltas} we add bump functions to $h$ to show that we can make $\mu_h$ approximate an arbitrary finite sum of point masses. Then in Proposition \ref{samemetric} we show that the metric is not greatly modified, by showing that geodesic can freely go around the added bumps without adding much length. The bump functions we add will be large on the union of a large collection of small squares, and very negative on a small neighborhood of each of these squares. See Section \ref{colorfulsquares} for a precise definition. This construction is similar to Section 11.2 in \cite{bg-harmonic-ball}. Together with the results in Section \ref{flatproof} this gives the proof of Theorem \ref{measandmet} for the case that our metric is a conformally flat Riemannian one. The last step is to use the following folklore statement.
\begin{theorem}
Let $\bar{d}$ be a length metric on $\R^d$ which induces the same topology as the Euclidean metric $d_0.$ Let $\e>0$ and $R>0.$ Then there exists a bounded continuous function $f:\R^d\to \R$ such that
\[
\sup_{x,y \in [-R,R]^d}\vert \bar{d}(x,y)-e^f\cdot d_0(x,y) \vert \leq \e.
\]
\end{theorem}
A proof of this is given in \cite{appendix}.

\section*{Acknowledgements}

We thank Dmitri Burago, Tuca Auffinger, and Sahana Vasudevan for helpful discussions. E.G.\ was partially supported by a Clay research fellowship and by National Science Foundation grant DMS-2245832. 

\section{Preliminaries}
 
Fix $\gamma\in (0,2)$. Let $  \bar{h}$ be a whole-space log-correlated Gaussian field as in~\eqref{eqn:cov}. Let $h$ be a \emph{log-correlated Gaussian field plus a continuous function}, i.e., a random generalized function on $\R^d$ which can be coupled with $h$ in such a way that $\bar{h} - h$ is a continuous function on $\R^d$. 
As explained in Section \ref{sec:lqg-def}, one can define the $\gamma$-LQG metric $D_{\bar{h}}$ and the LQG measure $\mu_{\bar{h}}$ associated with $\bar{h}.$

We will also define the sphere average of a function. Suppose that $g:{\R^d}\to \R$ is a function, $r>0,$ and $z_0\in{\R^d}.$ Then we define
\[
g_r(z_0):=\frac{1}{\vert B_r(z_0)\vert} \int_{\p B_r(z_0)}g(z)d\sigma(z) ,
\]
where $\sigma$ is the uniform measure on the sphere $\p B_r(z_0)$.

Let $h$ be a log-correlated Gaussian field, as in \eqref{eqn:cov}. For any $z\in {\R^d}$ and $r>0,$ we define the sphere average $h_r(z)$ as the average of $h$ over the sphere of radius $r$ centered at $z,$ which is well defined despite the fact that $h$ does not have well-defined pointwise values (see Section 11.1 in \cite{fgf-survey}).

We note for the future that the law of $h$ up to constant is invariant under complex affine transformations, that is for $a \in {\R^d}\setminus \{0\},$ $b\in {\R^d},$ we have 
\begin{equation}\label{spatialscaling}
h(a \cdot +b)-h_{\vert a\vert}(b) \overset{d}{=}  h.
\end{equation}

We will also use a few additional properties about log-correlated fields, such as an analogue of the white noise decomposition (subsection \ref{whitenoisesection}), and an analogue of the Cameron-Martin property for $h$ (subsection \ref{CMpropsection}): if $f : \R^d \to \BB C$ is a continuously differentiable function whose Dirichlet energy $\int_{\R^d} |\nabla f|^2 $ is finite and whose average $(f)_1(0)$ over the unit sphere is zero, then the law of $h+f$ is absolutely continuous with respect to the law of $h.$

\subsection{White noise decomposition for log-correlated Gaussian fields}\label{whitenoisesection}
Let $W$ be a space-time white noise on $\R^d \times \mathbb R_+$. That is, $W$ is the Gaussian random generalized function such that for each $f\in L^2 (\R^d \times \mathbb R_+)$, the (formal) integral $\int_0^\infty \int_{\R^d}  f(x,t) W(dx,dt)$ is centered Gaussian with variance $ \int_0^\infty \int_{\BB R^2} f(x,t)^2 \,dx\,dt$. For $t > 0$, we define the kernel $p_t(x)$ to be the kernel associated with $(-\Delta)^{\frac{d}{2}},$ that is $p_t(x)$ is formally the solution to 
\begin{equation}\label{kerneldef}
\left\{
\begin{matrix}
\p_t p_t + (-\Delta)^{\frac{d}{2}} p_t =0,\\
p_0=\delta_0.
\end{matrix}
\right.
\end{equation}
Rigorously, $p_t\in L^2(\R^d)$ is the unique kernel such that the (unique) solution to
\[
\left\{
\begin{matrix}
\p_t u + (-\Delta)^{\frac{d}{2}}u=0\\
u(0,x)=u_0(x)
\end{matrix}
\right.
\]
is given by
\[
u(t,x) = \int_{\R^d} p_t(x-y)u_0(y)dy.
\]
To see that this kernel is well defined, one can take the Fourier transform on both sides of the equation for $p_t$ to obtain
\[
\left\{
\begin{matrix}
\p_t \hat{p}_t + (2\pi\vert \xi\vert^2)^{\frac{d}{2}}\hat{p}_t = 0,\\
\hat{p}_0(\xi)= 1,
\end{matrix}
\right.
\]
which suggests we define
\[
p_t = \mathcal{F}^{-1} e^{-(2\pi \vert\xi\vert^2)^{\frac{d}{2}}t}.
\]
Note that by scaling, we have the identity
\begin{equation}\label{kernelscaling}
p_t(x) = \frac{1}{t}p_1\left(\frac{x}{t^{\frac{1}{d}}}\right)
\end{equation}
For $0  \leq \alpha \leq \beta < \infty$, we also define
\begin{equation}\label{whitenoisedecomp}
h_{\alpha,\beta}(y) :=\int_{\alpha^2 }^{\beta^2} \int_{{\R^d}} p_{t/2}(x-y ,t) W(dx,dt).
\end{equation}
We will need the following basic property about the kernel $p_t(x).$
\begin{lemma}\label{kernelconcatenation}
Let $p_t(x)$ be defined as in \eqref{kerneldef}. Then we have the following identity:
\[
\int_{\R^d} p_t(z_1-x)p_s(x-z_2) dx = p_{t+s}(z_1-z_2).
\]
\end{lemma}
\begin{proof}
By translation invariance, it suffices to show
\[
\int_{\R^d} p_{t}(x)p_{s}(x-z)dx = p_{s+t}(z).
\]
For this, suppose $u$ is  the solution to
\[
\left\{
\begin{matrix}
\p_t u + (-\Delta)^{\frac{d}{2}}u=0,\\
u(0,x)=u_0(x)
\end{matrix}
\right.
\]
where $u_0$ is any smooth compactly supported initial condition. Let
\[
v(t,x) = \int_{\R^d} \left(\int_{\R^d} p_t(z)p_s(x-y-z)dz\right)u_0(y)dy.
\]
Then
\begin{eqnarray*}
\p_t v(t,x) &=& \p_t \left(\int_{\R^d}\int_{\R^d} p_t(z)p_s(x-y-z) u_0(y)dydz\right)\\
&=&\int_{\R^d}\int_{\R^d} \p_t(p_t(z))p_s(x-y-z)u_0(y) dydz
\end{eqnarray*}
and so
\begin{equation}\label{ineq24}
\p_t v(t,x) = \int_{\R^d}\int_{\R^d} \p_t(p_t(z))p_s(x-y-z)u_0(y) dydz.
\end{equation}
On the other hand,
\begin{eqnarray*}
(-\Delta)^{\frac{d}{2}} v(t,x) &=& \int_{\R^d}\int_{\R^d} (-\Delta_x)^{\frac{d}{2}} p_t(x-y-z)(p_s(z))u_0(y)dydz\\
&=& \int_{\R^d}\int_{\R^d}\left((-\Delta_x)^{\frac{d}{2}}p_t\right)(x-y-z)p_s(z) u_0(y)dydz.
\end{eqnarray*}
Combining with \eqref{ineq24} we obtain
\[
\left(\p_t+(-\Delta)^{\frac{d}{2}}\right)v = \int_{\R^d}\int_{\R^d} \left(\p_t+(-\Delta)^{\frac{d}{2}}\right)p_t(z) p_s(x-y-z)u_0(y) dydz =0.
\]
Finally, note that
\begin{eqnarray*}
\lim_{t\to 0} v(t,x) &=& \lim_{t\to 0}\int_{\R^d}p_t(z)\left(\int_{\R^d}p_s(x-y-z)u_0(y)dy\right) dz\\
&=& \int_{\R^d} p_s(x-y)u_0(y)dy\\
&=& u(s,x).
\end{eqnarray*}
Therefore $u(s+\cdot,\cdot),v(\cdot,\cdot)$ both solve
\[
\left\{
\begin{matrix}
\p_t w + (-\Delta)^{\frac{d}{2}}w=0,\\
w(0,x)=u(s,x)
\end{matrix}
\right.
\]
This implies that $v(t,x) = u(s+t,x).$ Since the initial condition $u_0 \in C_c^\infty$ was arbitrary, This completes the proof.
\end{proof}

For $\alpha > 0$, one can check using the Kolmogorov continuity theorem that $h_{\alpha,\beta}$ is a continuous function. Indeed, for any two points $v,w$ we have
\begin{eqnarray*}
\mathrm{Var}(h_{\alpha,\beta}(v) - h_{\alpha,\beta}(w)) &=& 2(\mathrm{Var}(h_{\alpha,\beta}(v))-\mathrm{Cov}(h_{\alpha,\beta}(v),h_{\alpha,\beta}(w)))\\
&=& 2 \int_{\alpha^2}^{\beta^2} p_s(v,v) - p_s(v,w) ds \lesssim \left(\int_{\alpha^2}^{\beta^2} \norm{\nabla p_s}_{L^\infty}ds\right)\vert v-w\vert
\end{eqnarray*}
For $\alpha = 0$, we interpret $h_{0,\beta}$ as a random generalized function.

We collect some basic properties of the family $\{h_{0,\beta}\}_{\beta>0}$ we will need later on.
\begin{lemma}\label{staterg}
    The family $\{h_{0,\beta}:\beta > 0\}$ has the following properties:
    \begin{itemize}
    \item[-] If $\phi : \R^d \to \R^d$ is a rotation, reflection, or translation, then $\{h_{0,\beta} \circ \phi : \beta >0 \}$ has the same law as $\{h_{0,\beta} : \beta >0\}$.
    \item[-] If $\delta \in [0,1)$ we have that $\{  h_{0,\delta \beta} (\delta \cdot):\beta >0 \}$ and $\{h_{0,\beta}:\beta >0\}$ agree in law.
    \item[-] If $0\leq \beta_1\leq \beta_2\leq \beta_3\leq \beta_4 < \infty$ then $h_{0,\beta_2}-h_{0,\beta_1}$ and $h_{0,\beta_4}-h_{0,\beta_3}$ are independent. 
    \end{itemize}
    Finally, $h_{0,1}$ is ergodic with respect to translations of $\R^d$. 
\end{lemma}

The invariance and independence properties are immediate from the definition, see, e.g., Section 3.2 of \cite{ding-goswami-watabiki} for the two dimensional case. For the last statement, note that the white noise $W(dx,dt)$ is ergodic with respect to spatial translations, and therefore so is $h_{0,\beta}.$
An important consequence is the following scaling property of $D_{h_{0,\beta}}.$

\begin{lemma}\label{whitenoiserescaling}
Let $\beta>0$. Then $\{D_{h_{0,\beta}}(x,y) : x,y\in\R^d\}$ has the same law as
\[
\left\{ \beta^{\xi Q} D_{h_{0,1}}(\beta^{-1} x,\beta^{-1} y) : x,y\in \R^d \right\}
\]
\end{lemma}
\begin{proof}
This is a direct consequence of the LQG coordinate change formula, Weyl scaling and Lemma \ref{staterg}. Indeed, by the LQG coordinate change formula, a.s.\ 
\[
D_{h_{0,1}}(\beta^{-1} x,\beta^{-1} y) = D_{h_{0,1}(\beta^{-1} \cdot) + Q\log \left(\beta^{-1}\right) } (x,y) ,\quad \forall x,y\in\R^d .
\]
By Weyl scaling and Lemma \ref{staterg}, the metric on the right has the same law as $\beta^{-\xi Q}D_{h_{0,\beta}} $.
\end{proof}

\begin{lemma}\label{samelaw}
For each $\beta > 0$, $h_{\beta,\infty}$ is well-defined as a random function viewed modulo additive constant and agrees in law with $h \ast p_{\beta^2/2}$, modulo additive constant.
\end{lemma}
\begin{proof}
Let $f:\R^d\to \R^d$ be compactly supported and smooth such that $\int_{\R^d} f(z) = 0.$ Let $R>0$ be large. Then
\[
\mathrm{Var}\left(\int_{{\R^d}} f(z) h_{\beta,R}(z)dz\right) = \mathbb{E}\left(\left(\int_{\R^d} f(z) h_{\beta,R}(z)dz\right)^2\right)
\]
since
\[
\mathbb{E} \left(\int_{\R^d} f(z) h_{\beta,R}(z)dz \right)=0.
\]

Therefore,
\begin{eqnarray} \label{eqn:wn-var0}
&&\mathrm{Var}\left(\int_{\R^d} f(z) h_{\beta,R}(z)dz\right) \notag \\
&&= \mathbb{E}\left(\int_{\R^d}\int_{\R^d} f(x)f(y)h_{\beta,R}(x)h_{\beta,R}(y) dxdy\right) \notag \\
&&= \mathbb E\left( \int_{\R^d}\int_{\R^d}\int_{\R^d}\int_{\beta^2}^{R^2}\int_{\R^d}\int_{\beta^2}^{R^2} f(x)f(y) p_{t/2}(x-u)p_{s/2}(y-v) W(du,dt) W(dv,ds) dxdy \right)\notag  \\
&&=\mathbb E\left( \int_{\R^d}\int_{\R^d}\int_{\R^d}\int_{\beta^2}^{R^2}\int_{\R^d}\int_{\beta^2}^{R^2} f(x)f(y) p_{t/2}(x-y-u)p_{s/2}(-v) W(du,dt) W(dv,ds) dxdy \right)\notag  \\
&&=\int_{\R^d}\int_{\R^d} f(x)f(y) K_R(x-y)dxdy,
\end{eqnarray}
where 
\[
K_R(x):= \mathbb E\left( \int_{\R^d}\int_{\beta^2}^{R^2}\int_{\R^d}\int_{\beta^2}^{R^2} p_{t/2}(x-u)p_{s/2}(-v)W(du,dt)W(dv,ds) \right) .
\]
Using the definition of the white noise we obtain
\[
K_R(x)=\int_{\R^d}\int_{\beta^2}^{R^2} p_{t/2}(x-u)p_{t/2}(u)dtdu
\]
Using Lemma \ref{kernelconcatenation} we have
\[
\int_{\R^d} p_t(z_1-x)p_t(x-z_2) dx = p_{2t} (z_1-z_2),
\]
and so
\begin{equation} \label{eqn:kernel-time}
K_R(x)= \int_{\beta^2}^{R^2} p_t(x) dt.
\end{equation}
Let $\mathcal F$ denote the Fourier transform, that is
\[
\mathcal{F}(f)(\zeta)=\hat{f}(\zeta) := \int_{\R^d} e^{-2\pi i x\cdot \zeta} f(x).
\]
Using Plancherel, \eqref{eqn:kernel-time}, and the fact that $\mathcal{F}(f\ast g)=\hat{f}\hat{g}$ we obtain
\begin{equation} \label{eqn:kernel-plancherel}
\int_{\R^d}\int_{\R^d} f(x)f(y)K_R(x-y)dxdy = \int_{\R^d} \hat{K}_R(\zeta) \vert \hat{f}(\zeta)\vert^2 d\zeta.
\end{equation}
Note that
\begin{equation}\label{eq102}
\hat{p}_t(\zeta) =e^{-(2\pi)^dt\vert \zeta\vert^d},
\end{equation}
and so in particular it is integrable in both space and time. Hence
\begin{equation} \label{eqn:kernel-fourier}
\lim_{R\to\infty} \hat{K}_R(\zeta) = \lim_{R \to \infty} \int_{\beta^2}^{R^2} \hat{p}_t(\zeta) dt = \int_{\beta^2}^\infty \hat{p}_t(\zeta) dt= \int_{\beta^2}^\infty e^{-2^d\pi^dt\vert \zeta\vert^d} dt=\frac{1}{2^d\pi^d\vert \zeta\vert^d}e^{-\frac{2^d\pi^d\beta^2\vert \zeta\vert^d}{2}}.
\end{equation}
By plugging~\eqref{eqn:kernel-plancherel} into~\eqref{eqn:wn-var0}, then using~\eqref{eqn:kernel-fourier}, we obtain
\begin{equation} \label{eqn:wn-var}
\lim_{R\to\infty} \mathrm{Var}\left(\int_{\R^d} f(z) h_{\beta,R}(z)dz\right) 
=  \int_{\R^d} \frac{1}{2^d\pi^d\vert \zeta\vert^d}e^{-\frac{2^d\pi^d\beta^2\vert \zeta\vert^d}{2}}\vert \hat{f}(\zeta)\vert^2   d\zeta.
\end{equation}
Note that the above integral is well defined since $f$ is smooth and $f(0)=0,$ hence $\frac{\vert \hat{f}(\zeta)\vert^2}{\vert \zeta\vert^2}$ is integrable at $0.$ This shows that $f\mapsto \int_{\R^d} f(z) h_{\beta,\infty}(z) \,dz$ is well-defined as a centered Gaussian process on the set of smooth compactly supported test functions with integral zero. 
We will now show that this process agrees in law with $f\mapsto \int_{\R^d} f(z) h\ast p_{\beta^2/2}(z) \,dz$.

Recalling that $h_1(0) := \int_{\p B_{1}(0)}h d\sigma = 0$ and that $\int_{\R^d} f(z)dz = 0,$ we obtain
\begin{eqnarray}\label{lhseq}
&&\mathrm{Var} \left( \int_{\R^d} f(z) h\ast p_{\frac{\beta^2}{2}}(z) dz\right)\nonumber\\
&&=\mathrm{Var} \left( \int_{\R^d} f(z) h^0\ast p_{\frac{\beta^2}{2}}(z) dz\right)\nonumber\\
&&= \int_{\R^d}\int_{\R^d} f(x)f(y)\mathbb{E}(h^0\ast p_{\frac{\beta^2}{2}} (x) h^0\ast p_{\frac{\beta^2}{2}}(y)) dxdy\nonumber\\
&&= \int_{\R^d}\int_{\R^d} f(x)f(y) \left(\int_{\R^d} \int_{\R^d} p_{\frac{\beta^2}{2}} (x-u) p_{\frac{\beta^2}{2}}(y-v)\mathbb{E}(h^0(u)h^0(v))dudv\right) dxdy \nonumber\\
&& = \int_{\R^d} \int_{\R^d} f(x) f(y) \left( \int_{\R^d}\int_{\R^d} p_{\frac{\beta^2}{2}} (x-u) p_{\frac{\beta^2}{2}}(y-v) G(u,v) dudv\right) dxdy,
\end{eqnarray}
where $G(x,y)$ is given by \eqref{eqn:cov}. Note that by translations invariance we have
\begin{eqnarray*}
&&\int_{\R^d} \int_{\R^d} f(x)f(y) \left(\int_{\R^d} \int_{\R^d} p_{\beta^2/2}(x-u) p_{\beta^2/2}(y-v) \log(\max\{\vert u\vert,1\})dudv\right) dxdy\\ &&=  \int_{\R^d} \int_{\R^d} f(x)f(y) \left(\int_{\R^d} p_{\beta^2/2}(x-u) \log(\max\{\vert u\vert,1\})du\right)\left(\int_{\R^d} p_{\beta^2/2}(v) dv\right) dxdy\\
&&= \left(\int_{\R^d} f(y)dy\right) \int_{\R^d} f(x) \left(\int_{\R^d} p_{\beta^2/2}(x-u) \log(\max\{\vert u\vert,1\})du\right)\left(\int_{\R^d} p_{\beta^2/2}(v) dv\right) dx\\
&&=0,
\end{eqnarray*}
where in the last line we used that $\int_{\R^d} f=0.$ Therefore
\begin{eqnarray*}
&&\mathrm{Var} \left( \int_{\R^d} f(z) h\ast p_{\frac{\beta^2}{2}}(z) dz\right)\nonumber\\
&& = \int_{\R^d} \int_{\R^d} f(x) f(y) \left( \int_{\R^d}\int_{\R^d} p_{\frac{\beta^2}{2}} (x-u) p_{\frac{\beta^2}{2}}(y-v) \log\left(\frac{1}{\vert u-v\vert}\right) dudv\right) dxdy, 
\end{eqnarray*}
Let
\[
F_x(v):=\int_{\R^d} p_{\beta^2/2}(x-u)\log\left(\frac{1}{\vert u-v\vert}\right)du.
\]
Note that $F_x=p_{\beta^2/2}(x-\cdot) \ast \log\left(\frac{1}{\vert \cdot\vert}\right).$ Hence
\[
\hat{F}_x(\zeta) = \mathcal{F}\left(\log\left(\frac{1}{\vert \cdot\vert}\right)\right)\mathcal{F}({p}_{\beta^2/2}(x-\cdot))(\zeta) = \frac{1}{2^d\pi^d \vert \zeta\vert^d}\mathcal{F}({p}_{\beta^2/2}(x-\cdot))(\zeta) = \frac{1}{2^d\pi^d\vert \zeta\vert^d}\hat{p}_{\beta^2/2}(\zeta)e^{-2\pi i x \cdot \zeta}
\]
Then
\begin{eqnarray*}
&&\int_{\R^d} \int_{\R^d} f(x) f(y) \left( \int_{\R^d}\int_{\R^d} p_{\frac{\beta^2}{2}} (x-u) p_{\frac{\beta^2}{2}}(y-v) \log\left(\frac{1}{\vert u-v\vert}\right)dudv\right) dxdy\\
&&= \int_{\R^d}\int_{\R^d} f(x)f(y) \left(\int_{\R^d} F_x(v)p_{\beta^2/2}(y-v) dv\right)dxdy\\
&&=\int_{\R^d}\int_{\R^d} f(x)f(y) \left(\int_{\R^d}\hat{F}_x(\zeta)e^{2\pi iy\cdot \zeta} \bar{\hat{p}}_{\beta^2/2}(\zeta) d\zeta \right) dxdy \quad \text{(Plancherel)}\\
&&=\frac{1}{2^d\pi^d}\int_{\R^d}\int_{\R^d} f(x)f(y) \left(\int_{\R^d} \frac{1}{\vert \zeta\vert^d}e^{2 \pi i(y-x)\cdot\zeta} \vert \hat{p}_{\beta^2/2}(\zeta)\vert^2 d\zeta \right) dxdy\\
&&=\frac{1}{2^d\pi^d}\int_{\R^d} \vert\hat{f}(\zeta)\vert^2 \frac{1}{\vert \zeta\vert^d} \vert\hat{p}_{\beta^2/2}(\zeta)\vert^2d\zeta \quad \text{(Defn. of Fourier transform)} \\ 
&&=\int_{\R^d} \frac{\hat{f}(\zeta)^2}{2^d\pi^d\vert \zeta\vert^d} e^{-2^d\pi^d\beta^2\vert\zeta\vert^d}d\zeta.
\end{eqnarray*}
Combining this with \eqref{lhseq}, that
\[
\mathrm{Var} \left( \int_{\R^d} f(z) h\ast p_{\frac{\beta^2}{2}}(z) dz\right) =\int_{\R^d} \frac{\hat{f}(\zeta)^2}{2^d\pi^d\vert \zeta\vert^d} e^{-2^d\pi^d\beta^2\vert\zeta\vert^d}d\zeta.
\]
Now using \eqref{eqn:wn-var} we obtain
\[
\mathrm{Var} \left( \int_{\R^d} f(z) h\ast p_{\frac{\beta^2}{2}}(z) dz\right) = \lim_{R\to\infty} \mathrm{Var}\left(\int_{\R^d} f(z) h_{\beta,R}(z)dz\right) = \mathrm{Var}\left(\int_{\R^d} f(z) h_{\beta,\infty}(z)dz\right).
\]
Since the variances agree, this completes the proof.
\end{proof}

By sending $\beta \to 0$ in Lemma~\ref{samelaw}, we find that $h_{0,\infty}$ is well-defined as a random generalized function and agrees in law with $h$, modulo additive constant. We may therefore couple $h$ with the white noise $W$ in such a way that $\int_{\p B_1(0)} h_{\beta,\infty} =0,$ and
\begin{equation} \label{eqn:wn-decomp}
h = h_{0,\beta} + h_{\beta,\infty}  - (h_{0,\beta})_1(0),
\end{equation}
where we recall that $g_r(z_0)$ denotes the sphere average $\frac{1}{\vert B_r(z_0)\vert}\int_{\p B_r(z_0)} g(z) d\sigma (z).$ This decomposition will be crucial for our proofs.
We prove the following technical lemma, which will be used repeatedly in future calculations.
\begin{lemma}\label{avggoingto0}
We have that
\[
\left(h_{0,\beta}\right)_1(0) \to 0
\]
in probability as $\beta \to 0.$
\end{lemma}
\begin{proof}
Recall by \eqref{eqn:wn-var0} and \eqref{eqn:kernel-time} that
\[
\mathrm{Var}\left( \int_{\R^d} f(z) h_{\alpha,\beta}(z)dz\right) = \int_{\R^d}\int_{\R^d} f(x)f(y) K_{\alpha,\beta}(x-y)dxdy,
\]
where
\[
K_{\alpha,\beta}:=\int_{\alpha^2}^{\beta^2} p_t(z)dz.
\]
Note that the above inequality also holds for $f$ replaced by a measure $\mu,$ that is
\[
\mathrm{Var}\left( \int_{\R^d} h_{\alpha,\beta}(z)d\mu(z)\right) = \int_{\R^d}\int_{\R^d} K_{\alpha,\beta}(x-y)d\mu(x)d\mu(y).
\]
Again using Plancherel as in Lemma \ref{samelaw}, we obtain that
\[
\mathrm{Var}\left( \int_{\R^d} f(z) h_{\alpha,\beta}(z)d\mu(z)\right) = \int_{\R^d} \hat{K}_{\alpha,\beta}(\zeta) \vert \hat{\mu}(\zeta)\vert^2 d\zeta.
\]
Taking $\mu$ such that $\int_{\R^d} f(z) d\mu(z) = \int_{\p B_1(0)}f(z)d\sigma(z),$ we note that $\mu$'s Fourier transform is well defined. Therefore we have
\[
\mathrm{Var}\left( (h_{0,\beta})_1(0)\right) = \int_{\R^d} \lim_{\alpha\to 0}\hat{K}_{\alpha,\beta}(\zeta) \vert \hat{\mu}(\zeta)\vert^2 d\zeta,
\]
and similarly we have
\[
\mathrm{Var}\left( (h_{\beta,\infty})_1(0)\right) = \int_{\R^d} \lim_{R\to \infty}\hat{K}_{\beta,R}(\zeta) \vert \hat{\mu}(\zeta)\vert^2 d\zeta.
\]
We compute
\[
\lim_{\alpha\to 0}\hat{K}_{\alpha,\beta}(\zeta) = \lim_{\alpha \to 0} \pi\int_{\alpha^2}^{\beta^2} \hat{p}_t(\zeta) dt = \int_0^{\beta^2} \pi\hat{p}_t(\zeta) dt= \int_0^{\beta^2} \pi e^{-\frac{4\pi^2t\vert \zeta\vert^2}{2}} dt.
\]
Noting that this last expression converges to $0$ as $\beta \to 0,$ we see that
\[
\lim_{\beta \to 0} \mathrm{Var}\left( (h_{0,\beta})_1(0)\right) = 0.
\]
This completes the proof.
\end{proof}

\subsection{Cameron-Martin property for log-correlated Gaussian fields}\label{CMpropsection}

We will need the following property, analogous to the Cameron-Martin property for the GFF.
\begin{lemma}\label{CMhigherdim}
Suppose that $f \in C_0^\infty(\R^d)$ is a smooth compactly supported function. Then the law of $h+f$ is absolutely continuous with respect to the law of $h$ up to additive constant.
\end{lemma}
\begin{proof}
By Proposition 2.3 in \cite{imaginarymultchaos} we have the Karhunen–Lo\'eve expansion of $h$ (up to additive constant)
\[
h = \sum_{n \geq 0} A_n \sqrt{\lambda_n} \vphi_n(x)
\]
where $\{\lambda_n\}_{n \in \N}$ is the set of eigenvalues of the Hilbert-Schmidt operator associated with the kernel $p_t,$ $\{\vphi_n\}_{n \in \N}$ is an orthonormal basis of eigenvectors in $L^2,$ and $\{A_n\}_{n \geq 0}$ is an i.i.d. set of standard Gaussian random variables. Since $f \in C_c^\infty(\R^d),$ there exist coefficients $\{c_n\}_{n \in \N} \in \ell^2(\N)$ such that
\[
f = \sum_{n \in \N} c_n\sqrt{\lambda_n} \vphi_n.
\]
Indeed, since $f \in H_0^{\frac{d}{2}}(\R^d),$ we have
\[
f= \sum_{n \in \N} d_n \vphi_n
\]
for some coefficients $d_n,$ and so
\[
\infty> \int_{\R^d} \vert \Delta^{\frac{d}{4}}f(z)\vert^2 dz = \int_{\R^d} f(z) \Delta^{\frac{d}{2}}f(z)dz \geq \sum_{n\in \N} \vert d_n\vert \frac{\vert d_n\vert}{\lambda_n}
\]
and so $c_n = \frac{d_n}{\sqrt{\lambda_n}} \in \ell^2.$
Then
\[
h+f = \sum_{n \in \N} (A_n+c_n)\vphi_n.
\]
Therefore $h+f$ is absolutely continuous with respect to $h,$ with Radon-Nikodym derivative given by
\[
\prod_{n\in \N} e^{-\frac{c_n^2}{2}}e^{c_n x} = e^{-\frac{1}{2}\sum_{n \in \N} c_n^2 + \sum_{n \in \N} c_n A_n}
\]
which converges, since $c_n\in \ell^2.$ This completes the proof.
\end{proof}

\section{Aproximation theorem for Riemannian metrics}\label{flatproof}

In this section we will prove Theorem \ref{particular}.

In Section \ref{pfofsphere} we will then generalize this to the case of any measure on ${\R^d},$ still with the same metric.

We will need the following technical lemma about a moment bound for $\mathrm{diam}_{D_{h_{0,1}}} ([0,1]^d).$ In this we use the moment assumption \eqref{extramombound}.
\begin{lemma}\label{extramomboundlemma}
There exists a $p>d$ such that
\[
\mathbb{E}(\mathrm{diam}_{D_{h_{0,1}}}(B_1(0))^p) < \infty.
\]
\end{lemma}
\begin{proof}
Recall that $h = h_{0,1}+h_{1,\infty}.$ Therefore
\[
\mathrm{diam}_{D_{h_{0,1}}}([0,1]^d) \lesssim \mathrm{diam}_{D_h}([0,1]^d) e^{\xi \sup_{z \in [0,1]^d}h_{1,\infty}}.
\]
By the Borel TIS inequality, we have that $\sup_{z \in [0,1]^d}h_{1,\infty}$ has an exponential tail, and hence exponential moments of all orders. Therefore for any $d<p<\bar{p},$ we have
\[
\mathbb{E}(\mathrm{diam}_{D_{h_{0,1}}}([0,1]^d)^p) \lesssim \mathrm{E}(\mathrm{diam}_{D_h}([0,1]^d)^{\bar{p}})^{\frac{p}{\bar{p}}} \mathrm{E}\left(e^{p\left(\frac{\bar{p}}{p}\right)'\xi \sup_{z \in [0,1]^d}h_{1,\infty}}\right)^{\frac{1}{\left(\frac{\bar{p}}{p}\right)'}} < \infty
\]
where we used \eqref{extramombound} in the last line. This completes the proof.
\end{proof}
Theorem \ref{particular} will be proven by combining the following propositions.

\begin{prop} \label{approxeucdist} 
There exists $\nu > 0$ such that the following is true.
Let $\e>0.$ With probability going to $1$ as $\beta \to 0,$  
\begin{equation}\label{approxmetric}
\left\vert D_{h_{0,\beta}}(x,y)-\beta^{\xi Q-1}\nu\vert x-y\vert\right\vert\leq \e \beta^{\xi Q-1}
\end{equation}
for all $x,y\in [-R,R]^d.$ 
\end{prop}

\begin{prop} \label{approxeucmeas}
There exists a constant $C_0 > 0$ such that the following is true.
Let $c>0.$ With probability going to $1$ as $\beta \to 0,$ we have
\[
\mu_{h_{0,\beta}}([-R,R]^d) \leq C_0 \beta^{\gamma Q-2-c}.
\]
\end{prop}
Let $f\mapsto \hat f$ denote the Fourier transform and let $f\mapsto \check{f}$ denote the inverse Fourier transform,
\[
\check{f}(x)= \int e^{2\pi i \zeta \cdot x} d\zeta.
\]
We will also need the following lemma:
\begin{prop}\label{consthinfty}
Let $\e>0,$ and suppose that $f$ is a function such that $\hat{f}$ has compact support and $\int_{\p B_1(0)} f(z) d\sigma =0.$ For small enough $\beta > 0$, it holds with positive probability that
\[
\norm{h_{\beta,\infty}-f}_{L^\infty([-1,1]^2)}<\e.
\]
\end{prop}

We will now assume the results of Propositions \ref{approxeucdist}, \ref{approxeucmeas}, and \ref{consthinfty}, and prove Theorem \ref{particular}.

\subsection{Proof of Theorem \ref{particular}}
\label{sec-particular}

We have the following result, which follows from the same proof as Lemma 7.1 in \cite{df-lqg-metric}.
\begin{lemma}\label{limitmetric}
Let $U \subset {\R^d}$ be a bounded open set, and let $f:\bar{U} \to \R$ be a continuous function, and let $\vphi_1,\vphi_2:(0,\infty)\to (0,\infty)$ be increasing functions such that $\vphi_1(0^+)=\vphi_2(0^+)=0.$ Suppose $\{d_n\}_{n \geq 0}$ is a sequence of length metrics inducing the Euclidean topology on $\bar{U}$ such that for any $x,y\in \bar{U},$ and any $n\in\mathbb N$,
\[
\vphi_1(\norm{x-y}) \leq d_n(x,y) \leq \vphi_2(\norm{x-y}).
\]
Assume additionally that $d_n$ converges uniformly to a metric $d_\infty.$ Then $e^f \cdot d_n$ converges to $e^f \cdot d_\infty$ in the sense that for any $x,y \in \bar{U},$ we have
\[
\lim_{n \to \infty} e^f \cdot d_n(x,y) = e^f \cdot d_\infty(x,y).
\]
\end{lemma}
We will need the following corollary.

\begin{cor}\label{expl71}
Let $U \subset {\R^d}$ be a bounded open set, and let $f : \overline{U} \to \R$ be continuous. Let $\e > 0,$ and suppose that $\mathfrak{d}_0$ is a continuous length metric on $\overline{U}.$ Then there exists a $\delta > 0$ such that if $\mathfrak{d}$ is another continuous length metric on $\overline{U}$ such that $\sup_{x,y \in \overline{U}}\vert \mathfrak{d}_0(x,y)-\mathfrak{d}(x,y)\vert < \delta,$ then $\vert e^f\cdot \mathfrak{d}_0(x,y) - e^f\cdot \mathfrak{d}(x,y)\vert <\e$ for all $x,y\in \overline{U}$.
\end{cor}
\begin{proof}
If the conclusion does not hold, then there is a sequence of continuous length metrics $\mathfrak{d}^n$ such that
\begin{equation} \label{metriclim}
\sup_{x,y \in \overline{U}}\vert \mathfrak{d}_0(x,y)-\mathfrak{d}^n(x,y)\vert < \frac{1}{n},
\end{equation}
but 
\[
\sup_{x,y \in \overline{U}}\vert e^f\cdot \mathfrak{d}_0(x,y) - e^f\cdot \mathfrak{d}^n(x,y)\vert >\e.
\]
Therefore $\mathfrak{d}^n$ converges to $\mathfrak{d}_0$ uniformly. We claim that the hypotheses of Lemma \ref{limitmetric} hold, giving us a contradiction. For this we construct the functions $\vphi_1,\vphi_2.$ For any $r>0,$ let
\[
\bar{\vphi}_1(r):=\inf_{\substack{x,y:\vert x-y\vert=r\\ n \in \mathbb{N}}}\mathfrak{d}^n(x,y),
\]
and similarly
\[
\bar{\vphi}_2(r):=\sup_{\substack{x,y:\vert x-y\vert=r\\ n \in \mathbb{N}}}\mathfrak{d}^n(x,y).
\]
Finally, we define
\[
\vphi_1(r):=\inf_{s\geq r}\bar{\vphi}_1(s)
\]
and
\[
\vphi_2(r):=\sup_{s\leq r} \bar{\vphi}_2(s).
\]
Note that $\vphi_1(r), \vphi_2(r)$ are both decreasing, and $\bar{\vphi}_1(r)\geq \vphi_1(r),$ $\bar{\vphi}_2(r)\leq \vphi_2(r).$ Therefore by construction we have that
\[
\vphi_1(\vert x-y\vert)\leq \mathfrak{d}^n(x,y) \leq \vphi_2(\vert x-y\vert).
\]
It suffices to check now that $\vphi_1(r)>0$ if $r>0,$ $\vphi_2(0^+)=0,$ and $\vphi_2(s) < \infty$ for small enough $s.$ To check that $\vphi_1(r)>0,$ suppose by contradiction that this is not the case. Then there exist sequences $n_k,$ $x_k,y_k$ such that $\vert x_k-y_k\vert = r$ and
\[
\lim_{k\to \infty}\mathfrak{d}^{n_k}(x_k,y_k) = 0.
\]
Note that by passing to subsequence if necessary, we can assume that either $n_k=n$ does not depend on $k$ or $n_k \to \infty$ as $k \to \infty.$ In the first case, we would have that $\lim_{k\to \infty}\mathfrak{d}^n(x_k,y_k) = 0$ for some $n,$ which is impossible since $\mathfrak{d}^n$ is positive definite. In the second case, again by passing to a subsequence we can assume that $x_k \to x$ and $y_k\to y.$ Hence we have by \eqref{metriclim} that
\[
\lim_{k\to \infty}\mathfrak{d}^{n_k}(x_k,y_k) \geq \liminf_{k\to \infty} \mathfrak{d}_0(x_k,y_k) - \frac{1}{n_k} = \liminf_{k\to \infty}\mathfrak{d}_0(x_k,y_k) = \mathfrak{d}_0(x,y)>0,
\]
which yields a contradiction.

A similar contradiction argument shows that $\vphi_2(0^+)=0$ and $\vphi_2(s) < \infty.$ This completes the proof.
\end{proof}
We will need the following lemma about the behavior of metrics of the form $e^{F} \cdot \mathfrak{d}$ when $F$ is small.
\begin{lemma}\label{metricpert}
Let $F:{\R^d}\to{\R^d}$ be a continuous function, and suppose that $\mathfrak{d}$ is a length metric. Suppose that $F \in L^\infty({\R^d}),$ and that 
\[
\mathrm{Diam}_{\mathfrak{d}}([-R,R]^d):= \sup_{x,y\in [-R,R]^d} \mathfrak{d}(x,y) < \infty
\]
Then
\[
\sup_{x,y \in [-R,R]^d} \vert e^F \cdot \mathfrak{d}(x,y)-\mathfrak{d}(x,y)\vert \leq e^{\norm{F}_\infty}\left(e^{\norm{F}_\infty} -1\right) \mathrm{Diam}_{\mathfrak{d}}([-R,R]^d).
\]
\end{lemma}
\begin{proof}
This is an elementary consequence of the definition of $e^F\cdot \mathfrak{d}.$
\end{proof}

Now we will finish the proof of Theorem \ref{particular}.

\begin{proof}[Proof of Theorem \ref{particular}]
Now assume that we are in the setting of Theorem \ref{particular}.
Let $f:{\R^d}\to \R$ be a continuous function. From now on, we assume that the conclusions of Propositions \ref{approxeucdist}, \ref{approxeucmeas},
\begin{equation}\label{smallthing}
\vert (h_{0,\beta})_1(0)\vert \leq \e
\end{equation}
and \ref{consthinfty} hold simultaneously, which is a positive probability event for small enough $\beta>0.$ This is because the conclusions of Propositions \ref{approxeucdist} and \ref{approxeucmeas} hold with probability going to $1$ as $\beta \to0,$ \eqref{smallthing} holds with probability going to $1$ by Lemma \ref{avggoingto0}, the conclusion Proposition \ref{consthinfty} holds with positive probability, and also by the independence of $h_{0,\beta},h_{\beta,\infty}.$ We claim that
\begin{equation} \label{eqn:use-exp171}
\vert D_{h}(x,y)-\nu \beta^{\xi Q-1}e^{\xi f} \cdot d_0(x,y)\vert \leq \beta^{\xi Q-1}\e ,\quad \forall x ,y\in [-R,R]^d .
\end{equation}
This statement might seem odd at first glance, since $D_{h}(x,y)$ is of constant order while $\nu \beta^{\xi Q-1}e^{\xi f} \cdot d_0(x,y)$ is not. However, it holds because we are conditioning on the event that $\norm{h_{\beta,\infty}-f}_\infty<\e,$ which is rare, but holds with positive probability. To prove \eqref{eqn:use-exp171} we apply Corollary \ref{expl71} to the metrics $\mathfrak{d}_0=\nu d_0$ and $\mathfrak{d}=\beta^{1-\xi Q}D_{h_{0,\beta}}$ together with the function $f,$ which we know satisfy the hypotheses by Proposition \ref{approxeucdist} for small enough $\beta.$ Thus for small enough $\beta > 0,$ we have that for all $x,y\in [-R,R]^d,$
\[
\vert \beta^{1-\xi Q} D_{h_{0,\beta}+f}(x,y)-\nu e^{\xi f} \cdot d_0(x,y)\vert \leq \e,
\]
or equivalently,
\begin{equation}\label{randomineq}
\vert D_{h_{0,\beta}+f}(x,y)-\nu \beta^{\xi Q-1}e^{\xi f} \cdot d_0(x,y)\vert \leq \beta^{\xi Q-1}\e.
\end{equation}
We recall the white noise decomposition
\[
h = h_{0,\beta}+h_{\beta,\infty}-(h_{0,\beta})_1(0).
\]
This implies that
\[
\norm{h - (h_{0,\beta}+f)}_{\infty} = \norm{(h_{\beta,\infty}-f)-(h_{0,\beta})_1(0)}_{\infty} \leq \norm{h_{\beta,\infty}-f}_{\infty} + \vert (h_{0,\beta})_1(0)\vert.
\]
Using Proposition \ref{consthinfty} for the first term and \eqref{smallthing} for the second one, we see that for small enough $\beta$ we have
\[
\norm{h - (h_{0,\beta}+f)}_{\infty} \leq 2\e.
\]
Therefore by Lemma \ref{metricpert} we have that for all $x,y \in [-R,R]^d,$
\begin{equation}\label{ineq17}
\vert D_{h}(x,y)-D_{h_{0,\beta}+f}(x,y)\vert \leq e^{2\xi \e}\left(e^{2\xi \e}-1\right) \Diam_{D_{h_{0,\beta}+f}}([-1,1]^2).
\end{equation}
Applying Lemma \ref{metricpert} again for the metric $D_{h_{0,\beta}}$ and the perturbation $f,$ we see that
\[
\mathrm{Diam}_{D_{h_{0,\beta}+f}}([-R,R]^d) \leq e^{\xi \norm{f}_\infty} \mathrm{Diam}_{D_{h_{0,\beta}}}([-R,R]^d).
\]
Now using Proposition \ref{approxeucdist} we obtain that
\[
\mathrm{Diam}_{D_{h_{0,\beta}+f}}([-R,R]^d) \leq \beta^{\xi Q-1}\nu 2\sqrt{2} + \e \beta^{\xi Q-1}.
\]
Plugging this into \eqref{ineq17} we obtain that if $\e$ is small enough, for all $x,y\in [-R,R]^d$ we have
\[
\left\vert D_{h}(x,y)-D_{h_{0,\beta}+f}(x,y)\right\vert \leq \left(e^{2\xi \e} -1\right) \beta^{\xi Q-1} 4\nu.
\]
Combining this with \eqref{randomineq} and using the triangle inequality, we obtain that
\begin{eqnarray*}
\left\vert D_{h}(x,y)-\nu \beta^{\xi Q-1}e^{\xi f} \cdot d_0(x,y)\right\vert &\leq & \left\vert D_{h}(x,y)-D_{h_{0,\beta}+f}(x,y)\right\vert + \vert D_{h_{0,\beta}+f}(x,y)-\nu \beta^{\xi Q-1}e^{\xi f} \cdot d_0(x,y)\vert\\
&\leq & \beta^{\xi Q-1}\e+ \left(e^{2\xi \e} -1\right) \beta^{\xi Q-1} 4\nu.
\end{eqnarray*}
Therefore for small enough $\beta>0$ and shrinking $\e$ if necessary, we obtain our claim \eqref{eqn:use-exp171}.

Again recall we are assuming the conclusions of Propositions \ref{approxeucdist}, \ref{approxeucmeas} and \ref{consthinfty} hold simultaneously. Using the white noise decomposition $h=h_{p,\beta}+h_{\beta,\infty} - (h_{0,\beta})_1(0),$ we see that by Proposition \ref{consthinfty} and \eqref{smallthing} we have
\[
\norm{h - h_{0,\beta}-f}_\infty \leq 2\e.
\]
Hence
\begin{equation}\label{measineq}
\mu_{h}([-R,R]^d)\leq e^{\gamma \norm{f}_\infty+2\gamma\e} \mu_{h_{0,\beta}}([-R,R]^d)\leq C_0e^{\gamma \norm{f}_\infty+2\gamma\e} \beta^{\gamma Q-2-c}.
\end{equation}
This implies that if the conclusions of Propositions \ref{approxeucdist}, \ref{approxeucmeas}, and \ref{consthinfty} hold, then we have that for any $x,y \in [-R,R]^d,$
\[
\vert D_{h}(x,y)- \nu\beta^{\xi Q-1} e^{\xi f}\cdot d_0(x,y)  \vert \leq \beta^{\xi Q-1} \e,
\]
and
\[
\mu_{h}([-R,R]^d) \leq C_0 e^{\gamma \norm{f}_\infty+2\gamma\e} \beta^{\gamma Q-2-c}.
\]
Now replacing $f$ by $f+(\frac{1}{\xi}-Q)\log(\beta)-\frac{\log \nu}{\xi},$ we obtain
\[
\sup_{x,y\in [-R,R]^d}\vert D_{h}(x,y)- e^{\xi f}\cdot d_0(x,y)\vert\leq \e
\]
and for any Borel set $A,$
\[
\mu_{h}(A)\leq \theta \int_{A_\e} f(z) dz + \e
\]
where
\[
\theta = \alpha \nu^{-\frac{\gamma}{\xi}}\beta^{\frac{\gamma}{\xi} -d}.
\]
Note that by \eqref{gammaxiass}, $\theta\to 0$ as $\beta \to 0,$ and so this completes the proof of Theorem \ref{particular}.
\end{proof}

\subsection{First passage percolation argument}\label{fppargument}
From now on, let $\mbox{Diam}_{h} $ denote the diameter with respect to the metric $D_{h}$ and define $\mbox{Diam}_{h_{0,\beta}}$ analogously.
For the proof of this Proposition \ref{approxeucdist}, we first treat the case where $y$ is fixed. We will look at the asymptotic behaviour of $D_{h_{0,1}}(0,x)$ in a single direction first, and then prove convergence independent of the direction. To this end, we will need the following lemma.
\begin{lemma} \label{borelcantelli}
Let $c>0$ be a sufficiently small fixed constant. Then we have that almost surely,
\[
\sup_{\substack{x,y \in [-2^n,2^n]^d\\ \vert x-y\vert \leq 1}} D_{h_{0,1}}(x,y) \leq 2^{(1-c)n}
\]
for all but finitely many integers $n.$
\end{lemma}
\begin{proof}
We recall that $e(p):=\mathbb{E}\left(\left(\mathrm{Diam}_{D_{h_{0,1}}}([0,1]^d)\right)^p\right)<\infty$ for some $p>d$ by \eqref{extramombound}. By Markov's inequality,
\[
\mathbb{P}\left( \mathrm{Diam}_{D_{h_{0,1}}}([0,1]^d) > 2^{(1-c)n} \right) \leq \frac{C}{2^{np(1-c)}}.
\]
for some absolute constant $C$ only depending on $p.$ Now note that if $\sup_{\substack{x,y \in [-2^n,2^n]^d\\ \vert x-y\vert \leq 1}}D_{h_{0,1}}(x,y)> 2^{(1-c)n+d},$ then there must exist $x_0,y_0$ such that $\vert x-y\vert \leq 1,$ but $D_{h_{0,1}}(x,y)> 2^{(1-c)n+2}.$ This implies that there must exist a cube $[k_1,k_1+1]\times [k_2,k_2+1] \cdot \times [k_d,k_d+1]$ with $-2^n\leq k_1,k_2, \ldots , k_d \leq 2^n-1$ such that
\[
\mathrm{Diam}_{D_{h_{0,1}}}([k_1,k_1+1]\times [k_2,k_2+1] \cdot \times [k_d,k_d+1]) > 2^{(1-c)n}.
\]
Since there are $2^{dn}$ possibilities for the pair $(k_1,k_2),$ we obtain
\[
\mathbb{P}\left( \sup_{\substack{x,y\in [-2^n,2^n]^d\\ \vert x-y\vert \leq 1}}{D_{h_{0,1}}}(x,y) > 2^{(1-c)n} \right)\leq \left(2^n\right)^{d-p(1-c)}.
\]
Choosing $p$ to be close enough to $2$ such that $d-p(1-c)<0,$ we can apply Borel-Cantelli to obtain that almost surely, for all but finitely many $n$ we have
\[
\sup_{\substack{x,y \in [-2^n,2^n]^d\\ \vert x-y\vert \leq 1}} D_{h_{0,1}}(x,y) \leq 2^{(1-c)n}.
\]
This completes the proof.
\end{proof}

\begin{lemma} \label{nuexist}
Let $x\in {\R^d} \setminus \{0\}$ be fixed. There is a deterministic constant $\nu > 0$, not depending on $x$, such that a.s.
\[
\lim_{\lambda \to\infty}\left\vert \frac{D_{h_{0,1}}(0,\lambda x)}{\lambda\vert x\vert} - \nu \right\vert = 0.
\]
\end{lemma}
\begin{proof}
Recall that $h_{0,1}$ is stationary and ergodic by Lemma \ref{staterg} and therefore so is $D_{h_{0,1}}(0,x),$ which together with subadditivity (which is the triangle inequality for $D_{h_{0,1}}$) and Kingman's ergodic theorem (see \cite{kingman}) shows that there exists a deterministic function $\nu:\mathbb{S}^1 \to \R$ such that
\[
\lim_{n \to \infty} \left\vert \frac{D_{h_{0,1}}(0,n x)}{\lambda\vert x\vert}-\nu\left(\frac{x}{|x|} \right) \right\vert =0.
\]
Combining this with Lemma \ref{borelcantelli} we see that
\[
\lim_{\lambda \to \infty} \left\vert \frac{D_{h_{0,1}}(0,\lambda x)}{\lambda\vert x\vert}-\nu\left(\frac{x}{|x|} \right) \right\vert =0,
\]
where here $\lambda$ is a real number. Now all that remains is to see why $\nu$ is in fact constant. Simply note that the law of $D_{h_{0,1}}(0,x)$ is rotationally invariant with respect to $x$ by Lemma \ref{staterg} and hence $\nu$ is independent of the direction.
\end{proof}

\begin{lemma} \label{fixedcenter}
Almost surely, $\limsup_{\vert x\vert \to \infty, x\in {\R^d}} \left\vert \frac{D_{h_{0,1}}(0,x)}{\vert x\vert}-\nu\right\vert =0.$
\end{lemma}
For the proof of this lemma, we follow the argument used in the main theorem in \cite{indepperc}. The argument is almost exactly the same, but we reproduce it here for the reader's convenience.
First we will outline the proof. The first crucial element is the maximal lemma, Lemma \ref{max} below, which gives an upper tail bound for the random variable $\sup_{|x| > 1} D_{h_{0,1}}(0,x) / |x|$. The next step is Lemma \ref{claim}, which tells us that given any direction $\theta$ there is a $z'$ in that direction at a far enough distance from the origin such that $D_{h_{0,1}}(x,z')<\lambda \vert x\vert$ for a large constant $\lambda > 0$ which does not depend on $x.$

Now if the lemma is not true, then we must have a sequence $x_i \to \infty$ such that
\[
\left\vert \frac{D_{h_{0,1}}(0,x_i)}{\vert x_i \vert} - \nu \right\vert > \bar{\e}, \quad \forall i \in \mathbb N  .
\]
By extracting a subsequence, we can assume that $\frac{x_i}{\vert x_i\vert} \to y.$ We can then compare the distances $D_{h_{0,1}}(0,x_i)$ and $D_{h_{0,1}}(0,n_iMz)$ where $z \in \R^d$ with $|z|=1$ is an appropriate direction, $M$ is a large integer,
and $n_i \in \mathbb N$ is such that $\vert x_i\vert,\vert n_i Mz\vert$ are comparable.
Combining everything we can conclude that $\left| \frac{D_{h_{0,1}}(0,x_i)}{\vert x_i\vert}-\nu  \right| < \bar{\e},$ yielding a contradiction.

First we claim that that the probability of $D_{h}(0,x)/\vert x\vert$ being large is small. More precisely, we need to show the following ``maximal lemma”:
\begin{lemma}\label{max}
For some fixed constant $K_1$ we have
\[
\BB P\left\{ \sup_{x\in {\R^d}, \vert x\vert> 1} \frac{D_{h_{0,1}}(0,x)}{\vert x\vert}>\lambda\right\} \leq \frac{K_1}{\lambda^d}
\]
\end{lemma}
\begin{proof}
Let $\mathcal{S}$ be the set of closed squares of side length $1$ with corners in $\mathbb{Z}^d.$ 
For each pair of squares $S,S' \in \mathcal S$ which share a side, we define the weight
\[
w(S,S') = \mathrm{Diam}_{h_{0,1}}(S) + \mathrm{Diam}_{h_{0,1}}(S'). 
\]
For each $x\in \R^d$, choose $S_x \in \mathcal S$ such that $x\in S$ (if $x$ lies on the boundary of a square, we make an arbitrary choice). 
For $x,y\in {\R^d}$, let 
\[
L(x,y) :=\inf_{P} \sum_{i=1}^{|P|} w(P_{i-1} , P_i) = \mathrm{Diam}_{h_{0,1}}(S_x) + \mathrm{Diam}_{h_{0,1}}(S_y) + 2\inf_P \sum_{i=1}^{|P|-1} \mathrm{Diam}_{h_{0,1}}(P_i) 
\]
where the infimum runs over all paths $P : P_0 , P_1,\dots,P_{|P|}$ of cubes with $P_0 = S_x$ and $P_{|P|} = S_y$.
By the triangle inequality, 
\[
\frac{D_{h_{0,1}}(0,x)}{\vert x\vert} \leq \frac{L(0,x)}{\vert x\vert}.
\]
Hence
\[
\BB P\left\{ \sup_{x\in {\R^d}, \vert x\vert> 1} \frac{D_{h_{0,1}}(0,x)}{\vert x\vert}>\lambda\right\} \leq \BB P\left\{ \sup_{x\in {\R^d}, \vert x\vert> 1} \frac{L(0,x)}{\vert x\vert}>\lambda\right\}.
\]
Note that $h_{0,1}$ is stationary and ergodic by Lemma \ref{staterg}, and therefore so is $\mathrm{Diam}_{h_{0,1}} S.$ 
Since $\mathrm{Diam}_{h_{0,1}} S$ has a finite $p$-th moment by Lemma \ref{extramomboundlemma}, we can apply the maximal lemma in \cite{indepperc} (also see \cite[Theorem 6]{otherboivin}) to see that
\[
\BB P\left\{ \sup_{x\in {\R^d}, \vert x\vert> 1} \frac{L(0,x)}{\vert x\vert}>\lambda\right\} \leq \frac{K_1}{\lambda^d}.
\]
\end{proof}

Let $E_\lambda$ be the event that
\[
\sup_{x\in {\R^d}, \vert x\vert> 1}\frac{D_{h_{0,1}}(y,x)}{\vert x\vert}\leq\lambda,
\]
and write $E_\lambda = E_\lambda(0).$ By Lemma \ref{max} we know that
\[
\BB P (E_\lambda) \geq 1-\frac{K_1}{\lambda^d}.
\]
Following the notation in \cite{indepperc}, for $M \in \mathbb N$, let
\[
V_M:=\left\{\frac{z}{M}:z\in \Z^d \right\} ,\quad  V:=\bigcup_{M\in\mathbb N} V_M , \quad  B:=\left\{ z\in V:\vert v\vert=1 \right\} 
\]
and for $k , \rho > 0$, define the partial cone
\[
C(k,\rho):=\left\{x\in \R^d:x_1^2 + \cdots + x_{d-1}^2 \leq x_d^2,0\leq x_d\leq \rho\right\}.
\]
If $\overrightarrow{\theta}$ is a unit vector, let $\overrightarrow{\theta}C(k,\rho)$ denote the cone $C(k,\rho)$ rotated so that its axis is $\R \overrightarrow{\theta}.$ If $f$ is a function from the set of metrics on $\R^d$ to $\R,$ let
\[
\mathrm{Av}(\overrightarrow{\theta},k,\rho)(f):= \frac{1}{\vert \overrightarrow{\theta}C(k,\rho)\vert}\sum_{y\in \overrightarrow{\theta} C(k,\rho)\cap \Z^d} f(D_{h_{0,1}}(\cdot + y,\cdot + y)).
\]
Then $\mathrm{Av}(\overrightarrow{\theta},k,\rho)(f) \to \mathbb{E}( f (D_{h_{0,1}}) )$ as $\rho \to \infty$ by the Birkhoff ergodic theorem. Therefore applying this to $f=1_{E_\lambda}$ we get that almost surely, for each fixed $\lambda,k >0$ and $\overrightarrow{\theta} \in \R,$ we have
\[
\mathrm{Av}(\overrightarrow{\theta},k,\rho) 1_{E_\lambda} \to \BB P(E_\lambda)
\]
as $\rho \to \infty.$ In particular, there is some $N_1=N_1(\overrightarrow{\theta},k,\lambda)$ such that for all $\rho \geq N_1$ we have
\begin{equation}\label{ineq}
\BB P(E_\lambda)-\frac{K_1}{\lambda^d} \leq \mathrm{Av}(\overrightarrow{\theta},k,\rho)(1_{E_\lambda}) \leq \BB P(E_\lambda)+\frac{K_1}{\lambda^d}.
\end{equation}
\begin{lemma} \label{claim}
Suppose that $\overrightarrow{\theta} \in \R^d$ is a fixed unit vector. Suppose that $K_1/\lambda^d<1/4.$ Then there exists a random variable $N_1$ depending on $k,\overrightarrow{\theta}, \lambda$ if $\rho>\max\{N_1,2k\}$ and $\rho'-\rho>\max\{2, K_2 \rho \lambda^{-d}\}$ for some large constant $K_2$ only depending on $K_1,$ then there is a $z'\in \Z^d \cap (\overrightarrow{\theta}C(k,\rho')\setminus \overrightarrow{\theta}C(k,\rho))$ such that
\[
\sup_{x\in \R^d, \vert x- z'\vert >1.} \frac{D_{h_{0,1}}(z',x)}{\vert x-z'\vert}\leq\lambda.
\]
\end{lemma}
\begin{proof}
We have
\[
1-\frac{2K_1}{\lambda^d} \leq \BB P(E_\lambda) -\frac{K_1}{\lambda^d} \leq \mathrm{Av}(\overrightarrow{\theta},k,\rho)(1_{E_\lambda}).
\]
\begin{figure}[ht!]
\begin{center}
\includegraphics[width=0.5\textwidth]{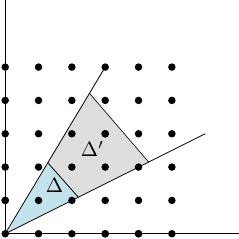}
\caption{\label{gridpic} The sets $\Delta$ and $\Delta'.$
}
\end{center}
\end{figure}
Let 
\[
\Delta:=\Z^d \cap \overrightarrow{\theta}C(k,\rho), 
\quad 
\Delta':=\Z^d \cap (\overrightarrow{\theta}C(k,\rho')\setminus \overrightarrow{\theta}C(k,\rho))
\]
and
\[
D=\Delta \cap \{ y :E_\lambda(y) \mbox{ holds}\},   \quad  D'=\Delta' \cap \{ y : E_\lambda(y) \mbox{ holds} \}
\]
(see Figure \ref{gridpic}). Then
\[
1-\frac{2K_1}{\lambda^d} \leq \frac{\vert D\vert+\vert D'\vert}{\vert \Delta\vert + \vert \Delta'\vert}.
\]
Hence
\[
\vert D'\vert 
\geq \vert \Delta'\vert + (\vert \Delta \vert - \vert D \vert )    -\frac{2K_1}{\lambda^d}(\vert \Delta\vert+\vert \Delta'\vert)   
\geq \vert \Delta'\vert-\frac{2K_1}{\lambda^d}(\vert \Delta\vert+\vert \Delta'\vert)
\]
To prove the claim it is enough to have $\vert D'\vert \geq 1$. By the above inequality, for this it suffices to have $(1-\frac{2K_1}{\lambda^d})\vert \Delta'\vert \geq \frac{2K_1}{\lambda^d}\vert \Delta\vert +1.$ Explicit calculation of $\vert \Delta\vert,\vert \Delta'\vert$ gives this inequality for large enough $\rho$ and $\rho'-\rho.$ 
\end{proof}
Now if Lemma \ref{max} is false, then we would have
\[
\BB P\left\{ \limsup_{x\in \R^d, \vert x\vert \to \infty} \left\vert \frac{D_{h_{0,1}}(0,x)}{\vert x\vert}-\nu \right\vert >0 \right\}>0.
\]
If this is the case, then by compactness of $\p B(0,1),$ we can find a sequence $\{x_i\}_{i\in\mathbb N}$ with $x_i \to \infty$ such that 
\begin{equation} \label{eqn:y}
\frac{x_i}{\vert x_i\vert} \to y
\quad \text{and} \quad 
\left\vert \frac{D_{h_{0,1}}(0,x_i)}{\vert x_i\vert} - \nu \right\vert>\bar{\e} ,\quad \forall i =1,2,\dots
\end{equation}
for some $\bar{\e}>0.$ Now we make a series of choices. Let $\lambda$ be large enough so that 
\begin{equation}\label{lambdak}
\lambda^{1-d}<\bar{\e}/(32\sqrt{d}K_2),
\end{equation}
\begin{equation}
\frac{2K_1}{\lambda^d} < \frac{1}{4}
\end{equation}
Also, take $M$ large enough so that there is a $z \in V_M  = \frac{1}{M} \mathbb Z^d$ with $\vert z\vert=1$ and 
\begin{equation}\label{distyz}
\vert y-z\vert<\bar{\e}/10\nu\lambda.
\end{equation}
For $i\in \mathbb N$, let $n_i$ be such that
\begin{equation}\label{ni}
n_iM\leq \vert x_i\vert<(n_i+1)M,
\end{equation}
and let $N_0$ be large enough so that for any $i>N_0,$
\begin{equation}\label{M/xi}
M/\vert x_i\vert <\bar{\e}/20\nu\lambda,
\end{equation}
\begin{equation}\label{xiconv}
\left\vert y-\frac{x_i}{\vert x_i\vert} \right\vert<\bar{\e}/10\nu\lambda,
\end{equation}
\begin{equation}\label{nudist}
\vert (n_iM)^{-1}D_{h_{0,1}}(0,n_iMz) -\nu \vert<\bar{\e}/10
\end{equation}
(which is possible by the ergodic theorem) and
\begin{equation} \label{nimz}
\norm{n_iMz}>\max\{2K_2,\lambda^2/K_2,N_1\}
\end{equation}
(see Figure \ref{nipic}).

\begin{figure}
\begin{center}
\includegraphics[width=0.35\textwidth]{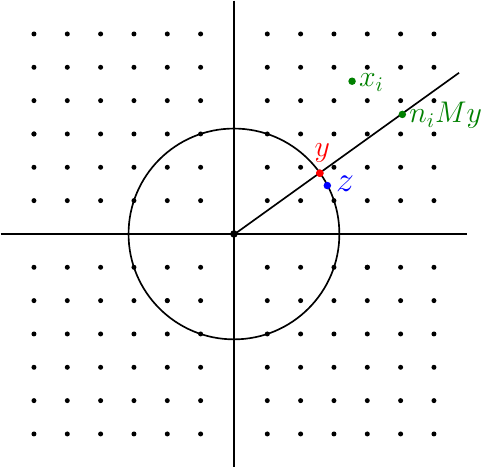}
\caption{\label{nipic} The $n_i,x_i$ and the renormalized limit $y.$
}
\end{center}
\end{figure}

\begin{lemma}
Let $\overrightarrow{\theta}\in \p B_1(0)$ be fixed. Then for large enough $k$ and $\rho,\rho'$ as in Lemma \ref{claim}, there exists a $z' \in \mathbb{Z}^d\cap \overrightarrow{\theta}(C(k,\rho')\setminus C(k,\rho))$ such that
\begin{equation} \label{distcomp}
\left\vert \frac{D_{h_{0,1}}(0,x_i)}{\vert x_i\vert}-\frac{D_{h_{0,1}}(0,n_iMz)}{\vert x_i\vert} \right\vert \leq \lambda\frac{(\vert x_i-n_iMz\vert+2\vert z'-n_iMz\vert)}{\vert x_i\vert}.
\end{equation}
\end{lemma}
\begin{proof}
By Lemma \ref{claim}, there is a $z' \in \Z^d \cap (\overrightarrow{\theta}C(k,\rho')\setminus \overrightarrow{\theta}C(k,\rho))$ with $\rho=\norm{n_iMz}$ such that
\begin{equation}\label{z'label}
\sup_{\substack{x\in \R^d\\ \vert x- z'\vert \geq 1}} \frac{D_{h_{0,1}}(z',x)}{\vert z'-x\vert}\leq \lambda.
\end{equation}
By repeated applications of the triangle inequality,
\begin{eqnarray}\label{distcomp'}
	\left\vert \frac{D_{h_{0,1}}(0,x_i)}{\vert x_i\vert}-\frac{D_{h_{0,1}}(0,n_iMz)}{\vert x_i\vert} \right\vert &\leq& \left\vert \frac{D_{h_{0,1}}(x_i,n_iMz)}{\vert x_i\vert} \right\vert\nonumber\\
	&\leq & \frac{(D_{h_{0,1}}(x_i,z')+D_{h_{0,1}}(z',n_iMz))}{\vert x_i\vert}\nonumber\\
	&\leq & \lambda \frac{(\vert x_i-z'\vert+\vert z'-n_iMz\vert)}{\vert x_i\vert}\nonumber\\
	&\leq & \lambda\frac{(\vert x_i-n_iMz\vert+2\vert z'-n_iMz\vert)}{\vert x_i\vert}.
\end{eqnarray}
\end{proof}

Recall $y$ from~\eqref{eqn:y}. We claim that 
\begin{equation}\label{contradiction}
\left\vert \frac{D_{h_{0,1}}(0,x_i)}{\vert x_i\vert} - \nu \right\vert < \bar{\e}
\end{equation}
which will contradict \eqref{eqn:y}. By the triangle inequality,
\begin{eqnarray}\label{comp1}
\vert x_i-n_iMz\vert &\leq& \vert x_i\vert \left( \left\vert \frac{x_i}{\vert x_i\vert}-y\right\vert + \left\vert y-\frac{n_iMz}{\vert x_i\vert}\right\vert\right)\nonumber\\
&\leq&\vert x_i\vert \left( \left\vert y-\frac{x_i}{\vert x_i\vert} \right\vert + \vert y-z\vert + \vert z\vert \left( 1-\frac{n_iM}{\vert x_i\vert} \right) \right)\nonumber\\
&=&\vert x_i\vert \left( \left\vert y-\frac{x_i}{\vert x_i\vert} \right\vert + \vert y-z\vert + \left( 1-\frac{n_iM}{\vert x_i\vert} \right) \right).
\end{eqnarray}
Let $z'\in \mathbb{Z}^d \cap (\overrightarrow{\theta}C(k,\rho')\setminus \overrightarrow{\theta}C(k,\rho))$ be as in \eqref{z'label} taking $\rho=\norm{n_iMz},\rho'-\rho=2K_2\rho/\lambda^d.$ Using \eqref{nimz} we obtain
\begin{eqnarray}\label{comp2}
\vert z'-n_iMz \vert &\leq& \sqrt{d}\left( \frac{\rho'}{k}+(\rho'-\rho)\right)\nonumber\\
&\leq & \sqrt{d}\left(2K_2 \norm{n_iMz}/k\lambda^d+ \frac{\norm{n_iMz}}{k} + \frac{2K_2 \norm{n_iMz}}{\lambda^d}\right)\nonumber\\
&\leq & \sqrt{d} \left( \frac{2 K_2}{k \lambda^d}+\frac{2 K_2}{\lambda^d} + \frac{1}{k}\right)n_i M \nonumber \\
&\leq & \sqrt{d}\left(\frac{4K_2}{\lambda^d} + \frac{1}{k}\right) n_iM.
\end{eqnarray}
In the first line we used the fact that $z',n_iMz \in e^{i\theta}(C(k,\rho')\setminus C(k,\rho)),$ and so it is easy to see that
\[
\vert z'-n_iMz\vert_1 \leq \sqrt{d}\vert z'-n_iMz\vert \leq \sqrt{d}\left(\frac{\rho'}{k} + (\rho'-\rho)\right) .
\]
The first line follows.
Plugging \eqref{comp1} and \eqref{comp2} into \eqref{distcomp}, we obtain
\begin{eqnarray*}
\left\vert \frac{D_{h_{0,1}}(0,x_i)}{\vert x_i \vert}-\frac{D_{h_{0,1}}(0,n_iMz)}{\vert x_i\vert} \right\vert \leq \lambda \left( \left\vert y-\frac{x_i}{\vert x_i \vert}\right\vert + \vert y-z\vert + \left( 1-\frac{n_iM}{\vert x_i\vert}\right)+ \sqrt{d}\left(\frac{4K_2}{\lambda^d}+\frac{1}{k} \right)\frac{n_iM}{\vert x_i\vert}\right).
\end{eqnarray*}
Now using \eqref{distyz},\eqref{M/xi}, and \eqref{xiconv} we obtain hat for large enough $\lambda, k$ we have
\begin{equation} \label{triangle0}
\left\vert \frac{D_{h_{0,1}}(0,x_i)}{\vert x_i \vert}-\frac{D_{h_{0,1}}(0,n_iMz)}{\vert x_i\vert} \right\vert \leq \lambda \left(  \frac{\bar{\e}}{10 \lambda}+\frac{\bar{\e}}{10\nu\lambda}+\frac{\bar{\e}}{20\nu\lambda}\right)<\frac{\bar{\e}}{2},
\end{equation}
where in the last line we used \eqref{lambdak}.
We also have by \eqref{ni} and \eqref{nudist} that
\begin{eqnarray}\label{triangle1}
\left\vert \frac{D_{h_{0,1}}(0,n_iMz)}{\vert x_i\vert} -  \frac{D_{h_{0,1}}(0,n_iMz)}{n_iM}\right\vert &\leq & \frac{M}{\vert n_iMx_i\vert} D_{h_{0,1}}(0,n_iMz)\leq \frac{M}{\vert x_i\vert} \left(\nu+\frac{\bar{\e}}{10}\right)\nonumber\\
&\leq& \frac{\bar{\e}}{20\nu\lambda}\left(\nu+\frac{\bar{\e}}{10}\right)\leq \frac{\bar{\e}}{10}
\end{eqnarray}
and
\begin{equation}\label{triangle2}
\left\vert \frac{D_{h_{0,1}}(0,n_iMz)}{n_iM} - \nu\right\vert \leq \frac{\bar{\e}}{10}.
\end{equation}
Finally, combining \eqref{triangle1}, \eqref{triangle2}, and \eqref{triangle0}, we obtain
\begin{eqnarray*}
\left\vert \frac{D_{h_{0,1}}(0,x_i)}{\vert x_i\vert} - \nu \right\vert &\leq & \left\vert \frac{D_{h_{0,1}}(0,x_i)}{\vert x_i\vert}-\frac{D_{h_{0,1}}(0,n_iMz)}{\vert x_i\vert}\right\vert\\
&+& \left\vert \frac{D_{h_{0,1}}(0,n_iMz)}{\vert x_i\vert}-\frac{D_{h_{0,1}}(0,n_iMz)}{n_iM}\right\vert\\
&+& \left\vert \frac{D_{h_{0,1}}(0,n_iMz)}{n_iM}-\nu\right\vert\\
&<& \bar{\e}
\end{eqnarray*}
which proves \eqref{contradiction}.

Now we are ready to prove Proposition \ref{approxeucdist}.

\begin{proof}[Proof of Proposition \ref{approxeucdist}]
Let $\Psi:(0,\infty)\to(0,\infty)$ be such that $\lim_{\beta\to 0} \Psi(\beta) = 0$ and $\lim_{\beta\to 0} \Psi(\beta)/\beta=\infty.$ Then we claim that for any fixed $\bar{x} \in \delta \mathbb{Z}^d$ we have with probability tending to $1$ as $\beta \to 0$ that
\[
\beta^{\xi Q-1}(1-\e) \nu\vert x-\bar{x}\vert \leq D_{h_{0,\beta}}(x,\bar{x}) \leq \beta^{\xi Q-1}(1+\e) \nu\vert x-\bar{x}\vert , \quad \forall x\in{\R^d} \: \text{such that} \: \vert x-\bar{x}\vert > \Psi(\beta) .
\] 
Indeed, in Proposition \ref{fixedcenter} $0$ could be replaced with any other fixed point $\bar{x}\in \mathbb{Z}^d,$ that is, almost surely we have that
\[
\limsup_{x \to \infty} 
\left\vert \frac{D_{h_{0,1}}(x,\bar{x})}{\vert x-\bar{x}\vert}-\nu\right\vert = 0. 
\]
By Lemma \ref{whitenoiserescaling} and Lemma \ref{staterg} we have that
\begin{eqnarray} \label{eqn:eucldistscale}
&&\BB{P}(\beta^{\xi Q-1}(1-\e) \nu\vert x-\bar{x}\vert \leq D_{h_{0,\beta}}(x,\bar{x}) \leq \beta^{\xi Q-1}(1+\e) \nu\vert x-\bar{x}\vert ,\: \forall x \: \text{s.t.} \: \vert x-\bar{x}\vert>\Psi(\beta)) \notag \\
&& = \BB{P}((1-\e) \nu\vert x\vert \leq D_{h_{0,1}}(x,0) \leq (1+\e) \nu\vert x\vert,\: \forall x \: \text{s.t.} \: \vert x\vert>\Psi(\beta)/\beta) \to 1
\end{eqnarray}
as $\beta \to 0.$ 

Now fix $\delta > 0$ and assume that $\beta$ is small enough that $\Psi(\beta) < \delta$. By a union bound, it holds with probability tending to $1$ as $\beta \to 0$ that for each $\bar{x}\in (\delta\mathbb{Z}^d)\cap[-R,R]^d,$ and all $x \in {\R^d}$ such that $\vert x-\bar{x}\vert> \delta,$ we have
\[
(1-\e)\nu \beta^{\xi Q-1} \vert x-\bar{x}\vert \leq D_{h_{0,\beta}}(\bar{x},x) \leq (1+\e)\nu \beta^{\xi Q-1} \vert x-\bar{x}\vert
\]
Now suppose that $x,y \in [-R,R]^d.$ Then there exists an $\bar{x}\in \delta \mathbb{Z}^d$ such that
\[
\delta < \vert x-\bar{x}\vert \leq C\delta
\]
for some $d$ dependent constant $d$ and
\[
\delta< \vert \bar{x}-y\vert.
\]
Therefore we obtain that
\begin{eqnarray*}
D_{h_{0,\beta}}(x,y) &\leq& D_{h_{0,\beta}}(\bar{x},x) + D_{h_{0,\beta}}(\bar{x},y) \leq (1+\e)\nu \beta^{\xi Q-1} (\vert x-\bar{x}\vert+\vert \bar{x}-y\vert)\\
&\leq& (1+\e)\nu \beta^{\xi Q-1} (2\vert x-\bar{x}\vert+\vert x-y\vert) \leq (1+\e)\nu \beta^{\xi Q-1} \vert x-y\vert + 2C\delta (1+\e)\nu \beta^{\xi Q-1}.
\end{eqnarray*}
Similarly,
\[
D_{h_{0,\beta}}(x,y) \geq (1-\e)\nu \beta^{\xi Q-1} \vert x-y\vert - 4\delta (1+\e)\nu \beta^{\xi Q-1}.
\]
Noting that $\e \nu \beta^{\xi Q-1} \vert x-y\vert + 4\delta (1+\e)\nu \beta^{\xi Q-1}$ can be made smaller than any desired multiple of $\beta^{\xi Q-1} |x-y|$ by making $\ep$ and $\delta$ small enough, we conclude the proof of Proposition \ref{approxeucdist}.
\end{proof}

\subsection{Controlling the total mass}

Let $\alpha$ be defined by
\[
\alpha = \mathbb{E} (\mu_{h_{0,1}}([0,1]^d))<\infty
\]
(which is well defined by Theorem 2.11 in \cite{rhodes-vargas-review}). Then by translation invariance, we have that for any positive integer $n,$
\[
\mathbb{E}(\mu_{h_{0,1}}([-2^n,2^n]^d)) \leq 2^d\alpha 2^{dn}.
\]
Therefore, by Markov's inequality we have that for any small $c>0,$
\[
\mathbb{P}(\mu_{h_{0,1}}([-2^n,2^n]^d)) \leq 2^d\alpha \left(2^n\right)^{d+c}) \geq 1- \frac{1}{2^{nc}} \to 0
\]
as $n \to \infty.$
By Lemma~\ref{staterg}, $h_{0,1}$ agrees in law with $h_{0,\beta}(\beta\cdot)$. By the Weyl scaling and LQG coordinate change properties of $\mu_{h_{0,1}}$, we deduce that $\mu_{h_{0,1}}([-2^n,2^n]^d)$ agrees in law with $\beta^{-\gamma Q}\mu_{h_{0,\beta}}([-\beta 2^n,\beta 2^n]^d)$. Hence picking $n$ such that $2^n\beta \in [1,11/10],$ we obtain with probability tending to 1 as $\beta \to 0$ that
\begin{equation}
\mu_{h_{0,\beta}}([-1,1]^d) \leq 4 \alpha \beta^{\gamma Q-d-c}.
\end{equation}

\subsection{Making \texorpdfstring{$h_{\beta,\infty}$}{hb,} approximate with positive probability}

We will need a couple of lemmas.
\begin{lemma}\label{invconv}
Let $f:{\R^d}\to\R$ be a smooth function and let $\beta>0.$ Let $p_{\beta^2/2}$ be as in \eqref{heatkernel}. Suppose that $f$'s Fourier transform $\hat{f}$ is compactly supported. Then there exists a $\phi$ with $\nabla \phi, \phi \in L^2({\R^d})$ such that
\[
f=p_{\beta^2/2}\ast \phi.
\]
\end{lemma}
\begin{proof}
Since $f$ is compactly supported, its Fourier transform is well defined. Let $\mathcal{F}^{-1}$ denote the inverse Fourier transform. Define
\[
\phi:=\mathcal{F}^{-1}\left( \frac{\hat{f}}{\hat{p}_{\beta^2/2}}\right).
\]
We note that by \eqref{eq102}, $\hat{p}_{\beta^2/2}$ never vanishes, and hence $\frac{\hat{f}}{\hat{p}_{\beta^2/2}}$ is well defined and by the compact support of $\hat f$ it is compactly supported. Therefore $\phi\in H^1({\R^d}).$ Also,
\[
\mathcal{F}(p_{\beta^2/2}\ast \phi) = \hat{p}_{\beta^2/2} \frac{\hat{f}}{\hat{p}_{\beta^2/2}}=\hat{f}.
\]
This completes the proof.
\end{proof}

The following lemma tells us that we can approximate uniformly by functions with compact Fourier support.
\begin{lemma}\label{fouriernet}
Let $R , \e > 0$. Then there exists a countable family of smooth functions such that for any $f\in \mathcal{F},$ the function $\hat{f}$ is compactly supported, and if $g$ is continuous in $\overline{B_R(0)}$ then there exists an element $f\in \mathcal{F}$ such that $\sup_{z\in B_R(0)}\vert g(z)-f(z)\vert<\e.$ In particular, all elements in this net are Schwartz.
\end{lemma}
\begin{proof}
Take $\tilde{\mathcal{F}}$ to be any countable $\e/2$-net of smooth functions in $\mathcal{C}(B_R(0))$. Let $\chi : {\R^d} \to [0,1]$ be a smooth function such that $\chi(y)=1$ for $\vert y\vert<1$ and $\chi(y)=0$ for $\vert y\vert\geq 2.$ Then for any $f\in \tilde{\mathcal{F}}$ and $N\in \N,$ let $f_N=\mathcal{F}^{-1}\left(\hat{f}\chi\left(\frac{\cdot}{N}\right)\right).$ Then $f_N$ has compactly supported Fourier transform, and $\norm{f_N-f}_\infty\leq\norm{\hat{f_N}-\hat{f}}_1 \to 0$ as $N \to 0.$ Now taking $\mathcal{F}=\{f_N: f\in\tilde{\mathcal{F}},N\in\N\}$ we have our desired net.
\end{proof}

Next, we need to construct functions such that the sphere average at the unit sphere can be arbitrarily large, but when convolved with $p_{\beta^2/2}$ is small.
\begin{lemma}\label{Psi}
Let $\beta>0,$ $K>0$ and $\eta>0.$ Then there is a smooth function $\Psi_{K,\eta}=\Psi_{K,\eta,\beta}$ such that
\begin{equation}\label{statement1}
\int_{\p B_1(0)} \Psi_{R,\eta} =2 \pi K,
\end{equation}
and
\begin{equation}\label{statement2}
\norm{\Psi_{K,\eta} \ast p_{\beta^2/2}}_{\infty} < \eta.
\end{equation}
\end{lemma}
\begin{proof}
Let $\bar{\eta}>0$ be a small constant to be chosen later. Let $\Psi_{K,\eta}$ be such that
\[
\Psi_{K,\eta}(z)=
\left\{
\begin{matrix}
K & \mbox{ if }z \in \p B_1(0),\\
0 & \mbox{ if } \vert z\vert \leq 1-\bar{\eta} \mbox{ or } \vert z\vert \geq 1+\bar{\eta}.
\end{matrix}
\right.
\]
Suppose additionally that $\Psi_{K,\eta}(z) \in [0,K]$ for all $z.$ Then it is clear that \eqref{statement1} holds. For \eqref{statement2}, note that $p_{\beta^2/2}(z) \lesssim \frac{1}{\beta^d}$ for any $z\in {\R^d}.$ Therefore
\[
\Psi_{K,\eta}\ast p_{\beta^2/2}(z_0)=\int_C \Psi_{K,\eta} (z) p_{\beta^2/2}(z_0-z) \lesssim \frac{1}{\beta^d} \int_{{\R^d}} \Psi_{K,\eta}(z) dz \leq \frac{1}{ \beta^d} R \left(\pi(1+\bar{\eta})^2-\pi(1-\bar{\eta})^2\right) = \frac{4K}{\beta^d} \bar{\eta}.
\]
Therefore if we take $\bar{\eta} \leq \frac{\beta^2 \eta}{4K}$ we obtain \eqref{statement2}.
\end{proof}
The following lemma tells us that the sphere average of $h\ast p_{\beta^2/2}$ at the unit sphere is small if $\beta$ is small enough.
\begin{lemma}\label{smallcircavg}
Let $\e>0.$ It holds with probability converging to $1$ as $\beta \to 0$ that
\[
\left\vert \left(h \ast p_{\beta^2/2}\right)_1(0)\right\vert \leq \e.
\]
\end{lemma}
\begin{proof}
We recall that $h$ is locally in $ H^{-s}$ for any $s>0$ by Proposition 2.7 in \cite{shef-gff}.
 Therefore $h \ast p_{\beta^2/2} \to h$ in $H^{-s}$ for any $s>0$ as $\beta \to 0.$ Therefore $(h \ast p_{\beta^2/2})_1(0) \to h_1(0)$ as $\beta \to 0.$ Since $h_1(0)=0,$ this completes the proof.
\end{proof}
We will also need the following lemma.
\begin{lemma}\label{pigeoninevent}
Let $R>0,$ and $\e, \delta, \theta>0.$ Suppose that $E$ is an event such that $\BB{P}(E)>0.$ Then there is a smooth, compactly supported function $\phi$ such that $\phi(z)=0$ if $\vert z\vert \geq R+\delta+\theta,$ and such that with positive probability,
\[
\text{$E$ occurs} \quad \text{and} \quad \sup_{x\in B_R(0)} \vert h_{\beta,\infty}(x)-\phi(x)\vert \leq \e.
\]
\end{lemma}
\begin{proof}
Let $F_R$ denote a family of functions compactly supported on $B_{R+\delta+\theta}(0)$ that form a countable $\e$-net in $ C(B_{R+\delta}(0))$ (with the $L^\infty$ norm). Then
\[
0<\BB{P}(E) \leq \sum_{\phi\in F_R} \BB{P} \left(E \mbox{ holds and } \sup_{x\in B_R(0)} \vert h_{\beta,\infty}(x)-\phi(x)\vert \leq \e \right),
\]
so there is some $\phi$ such that
\[
\BB{P} \left(E \mbox{ holds and } \sup_{x\in B_R(0)} \vert h_{\beta,\infty}(x)-\phi(x)\vert \leq \e \right) > 0.
\]
This completes the proof.
\end{proof}
Let $\norm{\cdot}_\infty$ denote the $L^\infty(B_R(0))$ norm for some large fixed $R.$ Let $E$ denote the event in Lemma \ref{smallcircavg}. Then by Lemma \ref{pigeoninevent}, there exists a smooth function $f_\beta$ such that
\begin{equation}\label{closefunctions}
\BB P\left(E \mbox{ holds and }\norm{h_{\beta,\infty}-f_\beta}_\infty<\frac{\e}{4}\right) > 0.
\end{equation}
Suppose from now on that $E$ holds. By Lemma \ref{invconv}, there exists a smooth function $g_\beta$ such that
\begin{equation}\label{gbeta}
p_{\beta^2/2} \ast g_\beta= f-f_\beta.
\end{equation}
By Lemma \ref{samelaw} we can couple $h$ and $h_{\beta,\infty}$ in such a way that
\[
h_{\beta,\infty}= h \ast p_{\beta^2/2} - (h \ast p_{\beta^2/2})_1(0)
\]
Then we have that
\begin{eqnarray*}
\norm{h_{\beta,\infty}-f}_{\infty} & \leq & \norm{h_{\beta,\infty}-(f-f_\beta)-f_\beta}_{\infty}\\
&\leq & \norm{h_{\beta,\infty}-p_{\beta^2/2}\ast g_\beta-f_\beta}_{\infty}\\
&\leq & \norm{(h-g_\beta)\ast p_{\beta^2/2}-f_\beta}_{\infty} + \norm{h_{\beta,\infty}-h \ast p_{\beta^2/2}}_{\infty}\\
&=& \norm{(h-g_\beta)\ast p_{\beta^2/2}-f_\beta}_{\infty} + \vert \left(h \ast p_{\beta^2/2} \right)_1(0)\vert,
\end{eqnarray*}
where in the last line we used Lemma \ref{samelaw}. Let $K = (g_\beta)_1(0),$ and define $\tilde{g}_\beta=g_\beta-\Psi_{K,\e},$ where $\Psi_{K,\e}$ is given as in Lemma \ref{Psi}. Then we have by the triangle inequality that
\begin{eqnarray}\label{longstring}
\norm{h_{\beta,\infty}-f}_{\infty} &\leq & \norm{(h-g_\beta)\ast p_{\beta^2/2}-f_\beta}_{\infty} + \vert\left(h \ast p_{\beta^2/2} \right)_1(0)\vert\nonumber\\
&\leq & \norm{(h-\tilde{g}_\beta)\ast p_{\beta^2/2}-f_\beta}_{\infty} + \vert\left(h \ast p_{\beta^2/2} \right)_1(0)\vert + \norm{(g_\beta-\tilde{g}_\beta)\ast p_{\beta^2/2}}_\infty\nonumber\\
&=& \norm{(h-\tilde{g}_\beta)\ast p_{\beta^2/2}-f_\beta}_{\infty} + \vert\left(h \ast p_{\beta^2/2} \right)_1(0)\vert + \norm{\Psi_{K,\e}\ast p_{\beta^2/2}}_\infty\nonumber\\
&\leq & \norm{(h-\tilde{g}_\beta)\ast p_{\beta^2/2}-f_\beta}_{\infty} + \vert\left(h \ast p_{\beta^2/2} \right)_1(0)\vert + \e
\end{eqnarray}
where in the last line we used Lemma \ref{Psi}. Now we claim that the event
\begin{equation}\label{claim2}
\{\norm{h \ast p_{\beta^2/2}-f_\beta}_{\infty} ,\vert\left((h+\tilde{g}_\beta)\ast p_{\beta^2/2}\right)_1(0)\vert \leq \e\}
\end{equation}
holds with positive probability. Suppose that \eqref{claim2} holds for now. Then we have by Lemma \ref{CMhigherdim} that it holds with positive probability that
\[
\{\norm{(h-\tilde{g}_\beta) \ast p_{\beta^2/2}-f_\beta}_{\infty}, \vert\left(h\ast p_{\beta^2/2}\right)_1(0)\vert \leq \e\}.
\]
Combining this with \eqref{longstring} we obtain that
\[
\norm{h_{\beta,\infty}-f}_{\infty}<3\e
\]
with positive probability.

It remains to show that \eqref{claim2} holds with positive probability. Note that we already have with positive probability that we simultaneously have
\[
\norm{h_{\beta,\infty}-f_\beta}_{\infty}<\frac{\e}{4}
\]
and
\[
\vert (h\ast p_{\beta^2/2})_1(0)\vert \leq \e.
\]
Then
\begin{eqnarray*}
\vert\left((h+\tilde{g}_\beta)\ast p_{\beta^2/2}\right)_1(0)\vert &\leq & \e + \vert\left(\tilde{g}_\beta\ast p_{\beta^2/2}\right)_1(0)\vert\\
&\leq & \e + \vert\left(g_\beta\ast p_{\beta^2/2}\right)_1(0)\vert + \vert\left(\Psi_{K,\e}\ast p_{\beta^2/2}\right)_1(0)\vert\\
&= & \e + \vert\left(f-f_\beta\right)_1(0)\vert + \vert\left(\Psi_{K,\e}\ast p_{\beta^2/2}\right)_1(0)\vert,
\end{eqnarray*}
where in the last line we used \eqref{gbeta}. Now recall that $f_1(0)=0$ by hypothesis. Hence by \eqref{closefunctions} we have
\[
\vert(f-f_\beta)_1(0) \vert = \vert(f_\beta)_1(0)\vert \leq \vert (h_\beta,\infty)_1(0)\vert+\pi \frac{\e}{4} \leq \e.
\]
This implies that
\[
\vert ((h+\tilde{g}_\beta)\ast p_{\beta^2/2})_1(0)\vert \leq 2\e+ \vert (\Psi_{K,\e}\ast p_{\beta^2/2})_1(0)\vert.
\]
Using Lemma \ref{smallcircavg} we obtain that \eqref{claim2} holds with positive probability.

\section{Increasing measure without changing metric} \label{pfofsphere}
Let $f:{\R^d}\to \R$ be a bounded continuous function. By Theorem \ref{particular}, with positive probability we have
\begin{equation}\label{eq100}
e^{\xi f}\cdot d_0(x,y)-\e <D_{h}(x,y)<e^{\xi f} \cdot d_0(x,y)+\e,
\end{equation}
for all $x,y \in [-R,R]^d$ and also
\begin{equation}\label{eq101}
\mu_{h}([-R,R]^d) < \e.
\end{equation}
\begin{defn}
We let $E^0$ denote the event that both \eqref{eq100} and \eqref{eq101} hold. Note that this is a positive probability event by Theorem \ref{particular}
\end{defn}
If we are able to change the measure without changing the metric very much, we can then use the results from Section \ref{flatproof} and prove Theorem \ref{measandmet} in the case of any Riemannian metric of the form $e^{\xi f}\cdot d_0$ together with any finite measure on $\R^d.$ Therefore the set of metric measure spaces that can be approximated by $(D_h,\mu_h)$ with positive probability and the closure of Riemannian metrics on $\R^d$ are the same.
Now by \cite{appendix}, this proves Theorem \ref{measandmet}.

\subsection{Proof of Theorem \ref{measandmet}}\label{colorfulsquares}

The idea will be to add bump functions to the log-correlated Gaussian field $h$ which will locally add mass at the centers of these bumps, but will not affect distances very much since geodesics can go around these bumps. To this end, suppose that $z_1, \ldots z_M$ is a set of points in $[-R,R]^d,$ and suppose that $C_1, \ldots, C_M$ are positive weights. Let $N,K$ be large positive integers, and let $r>0$ be small. For each $1 \leq i \leq M$ we will subdivide the square $z_i+[-r,r]^d$ into $2^d$ sets $X_{(j_1,\ldots j_d)}^{i,K}$ with $j_1, \ldots , j_d \in \{0,1\}^2$ defined as follows:
\[
X_{(j_1,\ldots ,j_d)}^{i,K}:= z_i + \frac{r}{2K}(j_1,\ldots ,j_d) + \left( \bigcup_{\ell_1=-K}^{K-1} \cdots \bigcup_{\ell_d=-K}^{K-1} \left(\frac{r\ell_1}{K}, \ldots \frac{r\ell_d}{K}\right)+\left[0,\frac{r}{2K}\right]^d \right)
\]
(see Figure \ref{phisupppic} for a picture). The idea of considering sets of this form is taken from the proof of Lemma 11.10 in \cite{bg-harmonic-ball}.


\begin{figure}
\begin{center}
\includegraphics[width=0.4\textwidth]{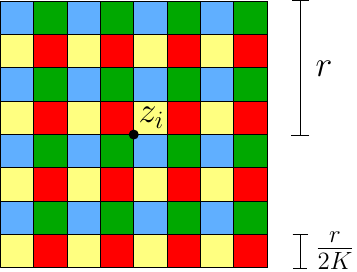}  
\caption{\label{phisupppic} The definition of $X_{(j_1,j_2)}^{i,K}.$ Different colors correspond to different $(j_1,j_2).$
}
\end{center}
\end{figure}

Let $\zeta>0$ be a small constant. Since $D_{h}$ induces the Euclidean topology, the maximum $D_h$ diameter of the cubes $X_{(j_1, \ldots , j_d)}^{i,K}$ goes to $0$ as $K \to \infty.$ By this and the pigeonhole principle, we can take deterministic $\{(j_1^i,\ldots ,j_d^i)\}_{1\leq i \leq M}$ such that with positive probability, $E^0$ holds, and simultaneously
\begin{equation}\label{sqrcond1}
\frac{\max_{-K \leq \ell_1,\ldots ,\ell_d\leq K} \mathrm{Diam}_{D_{h}}\left(z_i+\frac{r}{2K}(j_1^i,\ldots ,j_d^i)+\left(\frac{r\ell_1}{K},\ldots , \frac{r\ell_d}{K}\right)+\left[-\frac{r}{K},\frac{3r}{2K}\right]^d\right)}{\mu_{h}(z_i+[-r,r]^d)} \leq \zeta
\end{equation}
for all $1\leq i\leq M,$ and
\begin{equation}\label{sqrcond2}
\mu_{h}\left(X_{(j_1^i,\ldots , j_d^i)}^{i,K}\right) \geq \frac{\mu_{h}(z_i+[-r,r]^d)}{2^d}.
\end{equation}
To simplify notation, we let
\begin{equation}\label{XiKdef}
X_i^K:= X_{(j_1^i,\ldots , j_d^i)}^{i,K}.
\end{equation}
Now we will ``fix" the $\mu_{h}$ masses around each point $z_i.$ We have the following lemma.
\begin{lemma}\label{fixmass}
Suppose that $r>0$ and $\bar{\eta}>0.$ Then there exist integers $n_1, \ldots n_M \in \N,$ a deterministic $0< \rho < \frac{r}{K},$ and a large deterministic constant $\bar{C}$ such that with positive probability, $E^0$ holds, \eqref{sqrcond1} holds, \eqref{sqrcond2} holds, we have
\begin{equation}\label{sqrcond3}
\frac{\mu_{h}\left(\{z: d_0(z, X_i^K) \leq \rho\} \setminus X_i^K\right)}{\mu_{h}(X_i^K)} \leq \zeta,
\end{equation} 
also
\[
n_i\bar{\eta} \leq \mu_{h}(X_i^K)\leq (n_i+1)\bar{\eta}
\]
for any $1 \leq i \leq M,$ and we finally have
\begin{equation}\label{bound}
\mathrm{Diam}_{D_{h}}([-r-\rho,r+\rho]^d \setminus X_i^K) \leq \bar{C}.
\end{equation}
\end{lemma}
\begin{proof}
Note that the event that $E^0$, \eqref{sqrcond1}, \eqref{sqrcond2} hold simultaneously is an event of positive probability by construction. The result follows as a direct consequence of the pigeonhole principle. To select $\rho,$ we note that a non-deterministic $\rho$ satisfying \eqref{sqrcond3} exists with probability $1,$ and so there exists a deterministic $\rho$ satisfying \eqref{sqrcond3} simultaneously with the conditions $E^0$, \eqref{sqrcond1}, and \eqref{sqrcond2}. Finally, note that the probability of \eqref{bound} holds with probability going to $1$ as $\bar{C} \to \infty,$ so we can also assume this holds simultaneously. This completes the proof.
\end{proof}
Therefore for any $\bar{\eta} >0,$ there exist deterministic constants $\{a_i\}$ such that with positive probability, \eqref{sqrcond1}, \eqref{sqrcond2}, \eqref{sqrcond3} hold, and simultaneously
\begin{equation}\label{achoice}
\vert \mu_{h}(X_i^K) -a_i\vert\leq \bar{\eta}.
\end{equation}
We will define our bump functions now. For any $1 \leq i \leq M,$ we let $\vphi_{N,r,\rho}^{i}$ be defined as follows:
\begin{equation}\label{phidef}
\vphi_{N,r,\rho}^{i} (z)=
\left\{
\begin{matrix}
\log\left( \frac{C_i}{a_i} \right)& \mbox{ if } z\in X_i^K,\\
-N & \mbox{ if } d_0(x,X_i^K)\geq \rho \mbox{ and } z\in z_i + [-r-\rho,r+\rho]^d,\\
0 & \mbox{ if } d_0(z,z_i+[-r,r]^d) \geq 2\rho.
\end{matrix}
\right.
\end{equation}
We additionally impose that $-N \leq \vphi_{N,r,\rho}^{i} (z) \leq \log\left( \frac{C_i}{a_i} \right)$ if $d_0(x,X_i^K)\leq \rho,$ and that $-N \leq \vphi_{N,r,\rho}^{i} (z) \leq0$ if $d_0(z,z_i+[-r,r]^d) \leq \rho,$ but $z \notin z_i+[-r,r]^d$ (see Figure \ref{phisupppic2} for a picture of $\vphi_{N,r,\rho}^{i} (z)$ in dimension $2$).


\begin{figure}
\begin{center}
\includegraphics[width=0.4\textwidth]{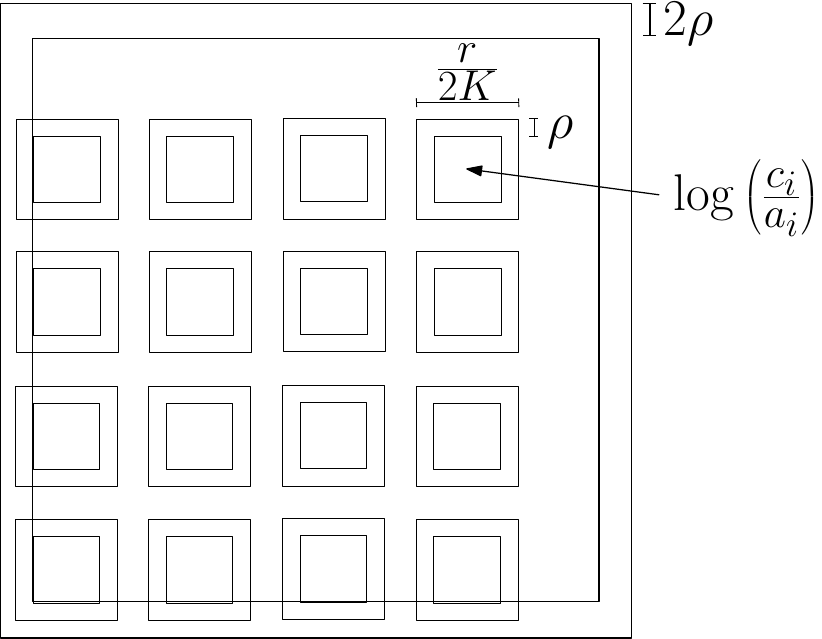}  
\caption{\label{phisupppic2} Support of $\vphi_{N,\e,\rho}^i$
}
\end{center}
\end{figure}

Now we define
\begin{equation}\label{vphidef}
\vphi_N (z) := \sum_{i=1}^M \vphi_{N,r,\rho}^{i} (z).
\end{equation}

We will prove the following proposition.
\begin{prop} \label{adddeltas}
Let $g$ be a continuous function on ${\R^d},$ and let $z_1, \ldots z_M\in {\R^d}.$ Define $\vphi_N$ as in \eqref{vphidef}. Suppose that $\int_{\R^d} |g(z)|\,dz < \infty,$ and suppose that the event of Lemma \ref{fixmass} occurs. Then for any $\zeta>0$ and small enough $\e,r,$
\[
\left\vert \int_{\R^d} g(z) e^{\vphi_N (z)} d\mu_{h} -\sum_{i=1}^M C_ig(z_i) \right\vert < \zeta
\]
\end{prop}
\begin{proof} 
First we note that we can reduce to the case that $M=1.$ Indeed, let $\{\chi_i\}_{i=1}^M$ be smooth functions such that
\[
\chi_i(z)=
\left\{
\begin{matrix}
1 & \mbox{ if } z \in z_i+[-r-2\rho,r+2\rho]^d,\\
0 & \mbox{ if } z \notin z_i+[-r-3\rho,r+3\rho]^d,
\end{matrix}
\right.
\]
and $0\leq \chi_i(z)\leq 1.$ Then by the triangle inequality we have
\begin{eqnarray*}
&&\left\vert \int_{\R^d} g(z) e^{\vphi_N (z)} d\mu_{h} -\sum_{i=1}^M C_ig(z_i)e^{\gamma f(z_i)} \right\vert \\
&&\leq \sum_{i=1}^M \left\vert \int_{\R^d} \chi_i(z)g(z) e^{\vphi_{N,r,\rho}^{i} (z)} d\mu_h -C_i\chi_i(z_i)g(z_i) \right\vert\\
&& + \left\vert \int_{{\R^d}} g(z)\left( e^{\vphi_N (z)} - \sum_{i=1}^M e^{\vphi_{N,r,\rho}^{i} (z)}\right)d\mu_{h}\right\vert\\
&& \leq \sum_{i=1}^M \left\vert \int_{\R^d} \chi_i(z)g(z) e^{\vphi_{N,r,\rho}^{i} (z)} d\mu_h -C_ig(z_i) \right\vert\\
&& + \e\norm{g}_\infty \max_{1\leq i \leq M} \vert C_i\vert,
\end{eqnarray*}
where in the last line we used the fact that $\mu_{h}({\R^d}) \leq \e$ by \eqref{eq101}. Therefore shrinking $\zeta$ if necessary, it suffices to prove the proposition for the case when $M=1.$ We then have a single point $z_1$ with its weight $C.$ Without loss of generality, suppose that $z_1=0.$ Since we have a single point mass, we will drop the subscript on $X_i^K.$ We split ${\R^d}$ into three regions,
\[
{\R^d}= X^K \cup \left((X^K + B_\rho(0))\setminus X^K\right) \cup \left( {\R^d} \setminus (X^K + B_\rho(0))\right).
\]
Splitting ${\R^d}$ into these domains we obtain
\begin{eqnarray*}
&&\left\vert \int_{\R^d} g(z) e^{\gamma f(z)+\vphi_N (z)} d\mu_h - Cg(0)e^{\gamma f(0)} \right\vert \\
&& \leq \left\vert \int_{X^K} g(z) e^{\vphi_N (z)} d\mu_h - Cg(0) \right\vert\\
&& + \left\vert \int_{\left((X^K + B_\rho(0))\setminus X^K\right)} g(z) e^{\vphi_N (z)} d\mu_h \right\vert\\
&&+ \left\vert \int_{\left( {\R^d} \setminus (X^K + B_\rho(0))\right)} g(z) e^{\vphi_N (z)} d\mu_h \right\vert.
\end{eqnarray*}
Note that
\[
\left\vert \int_{\left((X^K + B_\rho(0))\setminus X^K\right)} g(z) e^{\vphi_N (z)} d\mu_h \right\vert \leq \mu_{h}{\left((X^K + B_\rho(0))\setminus X^K\right)} \norm{g}_\infty \frac{C}{a},
\]
where we used the fact that $\vphi_N(z) \leq \frac{C}{a}$ if $z \in \left((X^K + B_\rho(0))\setminus X^K\right).$ Using \eqref{sqrcond3}, we obtain that
\begin{equation}\label{ineq22}
\left\vert \int_{\left((X^K + B_\rho(0))\setminus X^K\right)} g(z) e^{\vphi_N (z)} d\mu_h \right\vert \leq \norm{g}_\infty C \zeta.
\end{equation}
Similarly, we have with probability $1$ that
\begin{equation}\label{ineq23}
\left\vert \int_{\left( {\R^d} \setminus (X^K + B_\rho(0))\right)} g(z) e^{\vphi_N (z)} d\mu_h \right\vert \leq \norm{g}_\infty \mu_{h}({\R^d} \setminus (X^K + B_\rho(0))) \leq \norm{g}_\infty \e.
\end{equation}

\begin{eqnarray}\label{plugged}
\left\vert \int_{\R^d} g(z) e^{\vphi_N (z)} d\mu_h - Cg(0) \right\vert \leq \left\vert \int_{X^K} g(z)\frac{C}{a} - Cg(0) \right\vert  + \norm{g}_\infty  C \zeta + \norm{g}_\infty \e
\end{eqnarray} 
Note that
\begin{equation}\label{ineq21}
\int_{X^K} g(z)\frac{C}{a} - Cg(0) = o_r(1)
\end{equation}
by \eqref{achoice} and the fact that $g$ is continuous. Therefore taking $\e$ to be small enough we obtain
\begin{equation}
\left\vert \int_{\R^d} g(z) e^{\vphi_N (z)} d\mu_h - Cg(0) \right\vert \leq \zeta.
\end{equation}

\end{proof}

Now we claim the metric doesn't change much when adding $\vphi_N(z)$.

\begin{prop}\label{samemetric}
Let $\theta>0.$ For large enough $N,$ and small enough $r,$ we have that with positive probability that the conclusion of Lemma \ref{fixmass} holds, and simultaneously for all $x,y \in {\R^d}$
\[
\sup_{x,y\in{\R^d}} \vert D_{h}(x,y)-e^{\vphi_N}\cdot D_{h}(x,y)\vert \leq \theta.
\]
\end{prop}
We will prove Proposition \ref{samemetric} in Subsection \ref{metric}. We will prove now Theorem \ref{measandmet} assuming Proposition \ref{samemetric}.

\begin{proof}[Proof of Theorem \ref{flat} assuming Proposition \ref{samemetric}]
Suppose that $e^{\xi f} \cdot d_S$ is a Riemannian metric in its isothermal form, and let $\mathfrak{m}$ be a probability measure on ${\R^d}.$ We will show Theorem \ref{measandmet} for this metric and measure. Let $\vphi=\vphi_N$ be as in \eqref{vphidef}. We now introduce a finite set of functions to which we will apply Proposition \ref{samemetric}. Let $n$ be a large integer chosen so that
\begin{equation}\label{largen}
n\geq \frac{2\sqrt{d}R}{\e}.
\end{equation}
Let $T_{(j_1,\ldots , j_d)}$ be defined by
\[
T_{(j_1,\ldots , j_d)}:=\left[\frac{R}{n}(-n+j_1),\frac{R}{n}(-n+j_1+1)\right]\times \cdots \times \left[\frac{R}{n}(-n+j_d),\frac{R}{n}(-n+j_d+1)\right].
\]

For each finite subset $S \subseteq [0,2n-1]^d,$ let $g_S$ be a smooth function such that
\[
g_S(z) =
\left\{
\begin{matrix}
1 & \mbox{ if } z \in \cup_{(j_1, \ldots , j_d)\in S} T_{(j_1,\ldots , j_d)},\\
0 & \mbox{ if } z \notin \cup_{(j_1, \ldots , j_d)\in S} \{x : d_S(x,T_{(j_1,\ldots , j_d)}) \leq \frac{\e}{2}\}.
\end{matrix}
\right.
\]
For any $\theta>0,$ there exist points $\{z_i\}_{i=1}^n \subset {\R^d}$ and weights $\{C_i\}_{i=1}^n \subset (0,\infty)$ such that for any $(j_1, \ldots , j_d)\in S$ we have
\begin{equation}\label{measapprox}
\left\vert \int_{\R^d} g_S d\mathfrak{m}-\sum_{i=1}^M C_i g_S(z_i) \right\vert < \frac{\theta}{3} \left\vert \int_{\R^d} g_S d\mathfrak{m} \right\vert.
\end{equation}
Then applying Proposition \ref{adddeltas} to $g_S$ for every $S \subseteq [0,2n-1]^d$ taking $r$ to be small enough we obtain that with positive probability for any $S \subseteq F(\mathcal{T}),$ we have
\begin{equation}\label{ineq19}
\left\vert \int_{\R^d} g_S(z) e^{\vphi_N (z)} d\mu_h -\sum_{i=1}^M C_ig_S(z_i) \right\vert < \zeta,
\end{equation}
\[
e^{\xi f} \cdot d_0(x,y)-\e^{} \leq D_{h}(x,y)\leq e^{\xi f} \cdot d_0(x,y)+\e^{} ,\quad \forall x,y\in\mathbb{C}
\]
and simultaneously
\[
\mu_{h}({\R^d}) \leq \e.
\]
Therefore using \eqref{measapprox} we obtain
\begin{eqnarray*}
\left\vert \int_{\R^d} g_S(z) e^{\vphi_N(z)}d\mu_{h}-\int_{\R^d} g_Sd\mathfrak{m} \right\vert  
&\leq& \left\vert \int_{\R^d} g_S(z) d\mathfrak{m}-\sum_{i=1}^M C_i g_S(z_i) \right\vert\\
&+& \left\vert \int_{\R^d} g_S(z) e^{\vphi_N(z)}d\mu_{h} -\sum_{i=1}^M C_i g_S(z_i) \right\vert\\
&\leq & \frac{\theta}{3} \left\vert \int_{\R^d} g_S d\mathfrak{m} \right\vert + \zeta,
\end{eqnarray*}
where in the last inequality we used \eqref{measapprox} and \eqref{ineq19}.
Again, for large enough $N$ and small enough $r,$ applying Proposition \ref{samemetric} we obtain that for all $x,y\in{\R^d}$ we have
\begin{equation}\label{modded}
D_{h}(x,y)- \frac{\theta}{2} \leq e^{\frac{\xi}{\gamma}\vphi_N}\cdot D_{h}(x,y) \leq D_{h}(x,y) + \frac{\theta}{2}.
\end{equation}
Recall that we assumed that $z_i \notin \p B_1(0).$ Thus using the Cameron-Martin property applied to $\frac{1}{\gamma}\vphi_N$ we obtain that for any $\eta>0,$ we have with positive probability that
\begin{equation}\label{closeclose}
\sup_{x,y \in {\R^d}}\vert D_{h}(x,y) - e^{\xi f} \cdot d_0(x,y)\vert \leq \frac{\theta}{2},
\end{equation}
\begin{equation}
\left\vert \int_{\R^d} g_S(z) e^{\vphi_N(z)}d\mu_{h}-\int_{\R^d} g_Sd\mathfrak{m} \right\vert \leq \theta \left\vert \int_{\R^d} g_S d\mathfrak{m}\right\vert,
\end{equation}
and simultaneously
\[
\mu_{h}([-R,R]^d) \leq \e.
\]
To pass to a statement about measures of sets, suppose that $A$ is a measurable set, and let $S \subseteq [0,2n-1]^d$ such that $\norm{g_S-1_A}_1\leq \theta,$ and $\mathrm{Supp}g_S \subseteq A_\e.$ Then
\[
\left\vert \vert A\vert- \int_{\R^d} g_S\right\vert \leq \frac{\theta}{3}
\]
which implies that
\begin{equation}\label{unitmeas}
\left\vert \mathfrak{m}(A) -\int_{\R^d} 1_A(z) e^{\xi f(z)+\vphi(z)}d\mu_{h}\right\vert \leq \theta.
\end{equation}
This completes the proof.
\end{proof}

\subsection{Proof of Proposition \ref{samemetric} \label{metric}}
The first ingredient we need is the following lemma.

\begin{lemma} \label{smallsqrdiam}
Assume that the events in the conclusion of Lemma \ref{fixmass} hold. Let $\bar{\eta} > 0.$ Then for large enough $N,$ we have for all $1 \leq i \leq M,$ we have that
\[
\mathrm{Diam}_{e^{\frac{\xi}{\gamma}\vphi_N} \cdot D_{h}}(z_i+[-r-\rho,r+\rho]^d)\leq \bar{\eta}, \quad \mathrm{Diam}_{D_{h}}(z_i+[-r-\rho,r+\rho]^d) \leq \bar{\eta}.
\]
\end{lemma}
For now we will assume the validity of this lemma and prove Proposition \ref{samemetric}.

Suppose $x,y \in {\R^d},$ and let $P$ be a $e^{\frac{\xi}{\gamma} \vphi_N}\cdot D_{h}$-distance minimizing geodesic between $x,y.$ We claim that
\[
e^{\frac{\xi}{\gamma} \vphi_N}\cdot D_{h}(x,y) \leq \ell_{e^{\frac{\xi}{\gamma} \vphi_N}\cdot D_{h}} (P \setminus \bigcup_{i=1}^M (z_i+[-r-2\rho , r+ 2\rho]^d)) + M \sup_{1\leq i\leq M}\mathrm{Diam}_{e^{\frac{\xi}{\gamma} \vphi_N}\cdot D_{h}}(z_i+[-r-2\rho,r+2\rho]^d).
\]

Indeed, there exist intervals $[t_i,s_i]_{i=1}^M$ such that $P(t_0)=x,$ $P(s_M)=y,$ $P_{[t_\ell,s_\ell]} \cap \left(\bigcup_{i=1}^M X_i^K\right) = \varnothing,$ and additionally, for every $1 \leq i \leq M,$ $P(s_i),P(t_{i+1}) \in X_j^K$ for some $1 \leq j\leq M.$
Then 
\begin{eqnarray*}
e^{\frac{\xi}{\gamma} \vphi_N}\cdot D_{h}(x,y) &\leq& \ell_{e^{\frac{\xi}{\gamma} \vphi_N}\cdot D_{h}} (P \setminus \bigcup_{i=1}^M (z_i+[-r-2\rho , r+ 2\rho]^d)) + \sum_{i=1}^M \ell_{e^{\frac{\xi}{\gamma} \vphi_N}\cdot D_{h}}(P\vert_{[s_i,t_{i+1}]})\\
&\leq &\ell_{e^{\frac{\xi}{\gamma} \vphi_N}\cdot D_{h}} (P \setminus \bigcup_{i=1}^M (z_i+[-r-2\rho , r+ 2\rho]^d))\\
&+& M \sup_{1\leq i\leq M}\mathrm{Diam}_{e^{\frac{\xi}{\gamma} \vphi_N}\cdot D_{h}}(z_i+[-r-2\rho,r+2\rho]^d)
\end{eqnarray*}
This proves the claim. Now note that
\begin{eqnarray*}
e^{\frac{\xi}{\gamma} \vphi_N}\cdot D_{h}(x,y)
&\leq &\ell_{e^{\frac{\xi}{\gamma} \vphi_N}\cdot D_{h}} (P \setminus \bigcup_{i=1}^M (z_i+[-r-2\rho , r+ 2\rho]^d))\\
&+& M \sup_{1\leq i\leq M}\mathrm{Diam}_{e^{\frac{\xi}{\gamma} \vphi_N}\cdot D_{h}}(z_i+[-r-2\rho,r+2\rho]^d)\\
&\leq & \ell_{D_{h}} (P)\\
&+&M \sup_{1\leq i\leq M}\mathrm{Diam}_{e^{\frac{\xi}{\gamma} \vphi_N}\cdot D_{h}}(z_i+[-r-2\rho,r+2\rho]^d).
\end{eqnarray*}
Now applying Lemma \ref{smallsqrdiam} we conclude the upper bound for $D_{e^{\frac{\xi}{\gamma} \vphi_N}\cdot D_{h}}(x,y).$ The lower bound can be obtained in the exact same way interchanging the metrics $D_{e^{\frac{\xi}{\gamma} \vphi_N}\cdot D_{h}}$ and $D_{h}.$ This completes the proof of Proposition \ref{samemetric}.

For the proof of Lemma \ref{smallsqrdiam}, we will need the following lemma.
\begin{lemma}\label{thing}
Assume that the events in the conclusion of Lemma \ref{fixmass} hold. Suppose that $x \in X_i^K+[-\rho,\rho]^d$ for some $1 \leq i \leq M.$ Let $\bar{\e}>0.$ Then for large enough $K$ depending on $r,$ we have
\[
e^{\frac{\xi}{\gamma}\vphi_N}\cdot D_{h}(x, \p (X_i^K+[-\rho,\rho]^d)) \leq \bar{\e}.
\]
\end{lemma}
\begin{proof}
By \eqref{sqrcond2}, we have that
\[
\left\vert \frac{C_i}{a_i}\right\vert \leq \frac{2^d\vert 
C_i\vert}{\mu_{h}(z_i+[-r,r]^d)}.
\]
Therefore if $P$ is any path between $x$ and $\p X_i^K+[-\rho,\rho]^2$ of Euclidean length at most $\frac{r}{2K},$ we note that
\[
e^{\frac{\xi}{\gamma}\vphi_N}\cdot D_{h}(x, \p X_i^K+[-\rho,\rho]^d) \leq \ell_{e^{\frac{\xi}{\gamma}\vphi_N}\cdot D_{h}}(P) \leq \left\vert \frac{C_i}{a_i}\right\vert \frac{r}{2K} \leq \frac{2^d\vert 
C_i\vert}{\mu_{h}(z_i+[-r,r]^d)} \frac{r}{2K}.
\]
Now choosing $K$ large enough so that
\[
\frac{r}{2K} \leq \frac{\mu_{h}(z_i+[-r,r]^d)}{2^d\vert   
C_i\vert} \bar{\e}
\]
we conclude.
\end{proof}

Now we will prove Lemma \ref{smallsqrdiam}.
\begin{proof}[Proof of Lemma \ref{smallsqrdiam}]
Now suppose that $x,y \in z_i+[-r-\rho,r+\rho]^d.$ Then by the triangle inequality,
\[
D_{e^{\frac{\xi}{\gamma}\vphi_N} \cdot D_{h}}(x,y) \leq \mathfrak{d}_x+ \mathfrak{d}_y+ D_{e^{\frac{\xi}{\gamma}\vphi_N}\cdot D_{h}}(\bar{x},\bar{y}),
\]
where for any point $z \in z_i+[-r-\rho,r+\rho]^d,$ we define
\[
\mathfrak{d}_z:=
\left\{
\begin{matrix}
e^{\frac{\xi}{\gamma}\vphi_N}\cdot D_{h}(z, \p (X_i^K+[-\rho,\rho]^d)) & \mbox{ if } z \in X_i^K+[-\rho,\rho]^d,\\
0 & \mbox{ otherwise,}
\end{matrix}
\right.
\]
and where $\bar{z}$ is defined to be such that $e^{\frac{\xi}{\gamma}\vphi_N}\cdot D_{h}(z, \p (X_i^K+[-\rho,\rho]^d)) = e^{\frac{\xi}{\gamma}\vphi_N}\cdot D_{h}(z,\bar{z})$ if $z \in X_i^K+[-\rho,\rho]^d,$ and defined as $z$ otherwise (see Figures \ref{barz} and \ref{bothbars}).

\begin{figure}
\centering
\begin{minipage}{0.4\textwidth}
  \centering
  \includegraphics[width=0.9\textwidth]{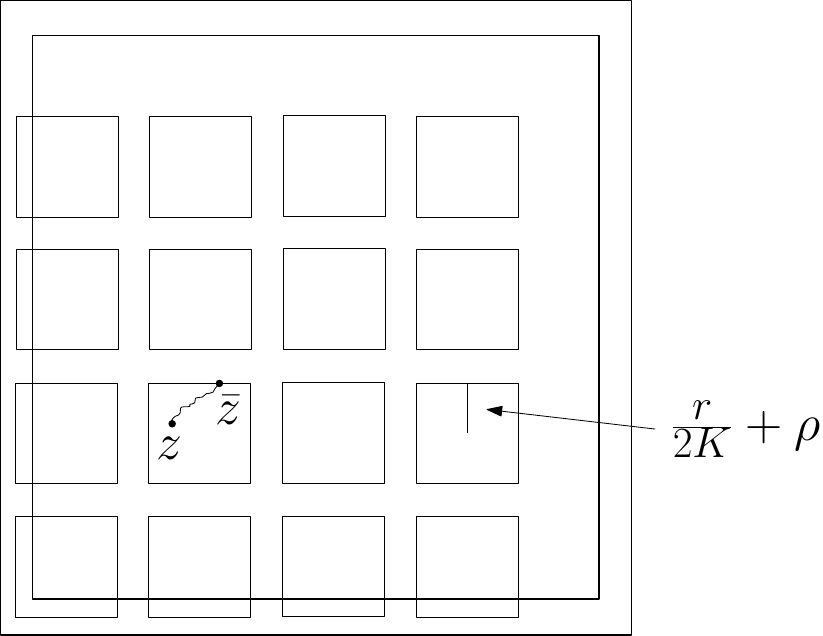}
  \caption{\label{barz} Definition of $\bar{z}$}
\end{minipage}
\begin{minipage}{0.4\textwidth}
  \centering
  \includegraphics[width=0.9\linewidth]{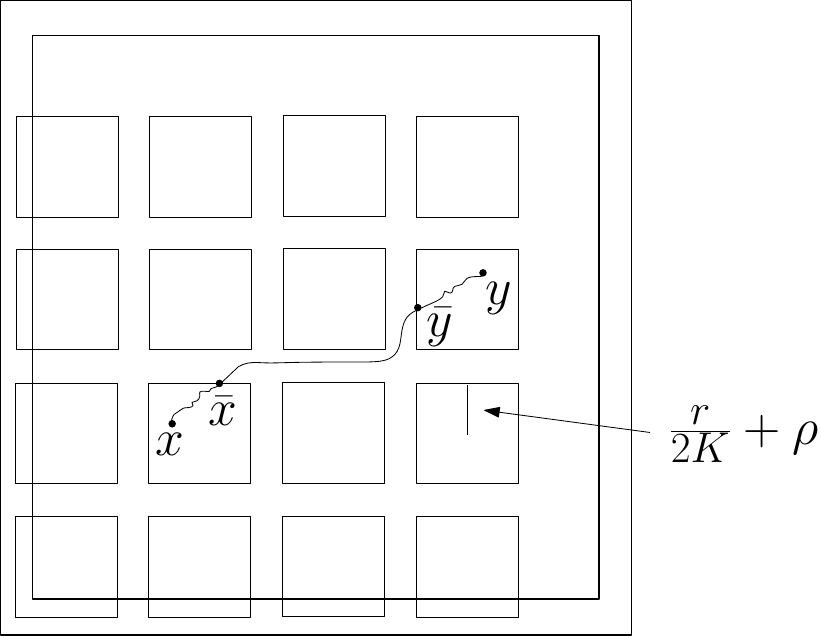}
  \caption{\label{bothbars} Triangle inequality argument}
\end{minipage}
\end{figure}

By Lemma \ref{thing}, for small enough $K$ we have that $\mathfrak{d}_x,\mathfrak{d}_y \leq \bar{\e}.$ Additionally, we note that by taking a path $P$ between $\bar{x},\bar{y}$ not intersecting $X_i^K+[-\rho,\rho]^d$ of $D_{h}$-length at most $M,$
\[
e^{\frac{\xi}{\gamma}\vphi_N}\cdot D_{h}(\bar{x},\bar{y}) \leq e^{-\frac{\xi}{\gamma}N}M,
\]
and so for a small enough choice of $\bar{\e}$ and $N$ large enough so that $e^{-\frac{\xi}{\gamma}N}M \leq \bar{\eta}$ we conclude that
\[
e^{\frac{\xi}{\gamma}\vphi_N}\cdot D_{h}(x,y) \leq 2 \bar{\e}+ e^{-N}r \leq \bar{\eta}.
\]
This completes the proof.
\end{proof}

\bibliography{ref}

\newcommand{\etalchar}[1]{$^{#1}$}
\begin{thebibliography}{DKRV16}

\bibitem[AA21]{aa-odd-angulation}
L.~{Addario-Berry} and M.~{Albenque}.
\newblock {Convergence of odd-angulations via symmetrization of labeled trees}.
\newblock {\em {A}nnales {H}enri {L}ebesgue}, 4:653--683, 2021, \arxiv{1904.04786}.

\bibitem[ABA17]{ab-simple}
L.~Addario-Berry and M.~Albenque.
\newblock The scaling limit of random simple triangulations and random simple quadrangulations.
\newblock {\em Ann. Probab.}, 45(5):2767--2825, 2017, \arxiv{1306.5227}. \MR{3706731}

\bibitem[AHS17]{ahs-sphere}
J.~Aru, Y.~Huang, and X.~Sun.
\newblock Two perspectives of the 2{D} unit area quantum sphere and their equivalence.
\newblock {\em Comm. Math. Phys.}, 356(1):261--283, 2017, \arxiv{1512.06190}. \MR{3694028}

\bibitem[BBI01]{bbi-metric-geometry(dupe)}
D.~Burago, Y.~Burago, and S.~Ivanov.
\newblock {\em A course in metric geometry}, volume~33 of {\em Graduate Studies in Mathematics}.
\newblock American Mathematical Society, Providence, RI, 2001. \MR{1835418}

\bibitem[BD91]{otherboivin}
D.~Boivin and Y.~Derriennic.
\newblock The ergodic theorem for additive cocycles of {$\mathbb{Z}^d$} or {$\mathbb{R}^d$}.
\newblock {\em Ergodic Theory and Dynamical Systems}, 11(1):19--39, 1991.

\bibitem[Ber17]{berestycki-gmt-elementary}
N.~Berestycki.
\newblock An elementary approach to {G}aussian multiplicative chaos.
\newblock {\em Electron. Commun. Probab.}, 22:Paper No. 27, 12, 2017, \arxiv{1506.09113}. \MR{3652040}

\bibitem[BG22]{bg-harmonic-ball}
A.~{Bou-Rabee} and E.~{Gwynne}.
\newblock {Harmonic balls in Liouville quantum gravity}.
\newblock {\em ArXiv e-prints}, August 2022, \arxiv{2208.11795}.

\bibitem[BG24]{bg-rw-sphere-packing}
A.~{Bou-Rabee} and E.~{Gwynne}.
\newblock {Random walk on sphere packings and Delaunay triangulations in arbitrary dimension}.
\newblock {\em ArXiv e-prints}, May 2024, \arxiv{2405.11673}.

\bibitem[BJM14]{bjm-uniform}
J.~Bettinelli, E.~Jacob, and G.~Miermont.
\newblock The scaling limit of uniform random plane maps, {\it via} the {A}mbj\o rn-{B}udd bijection.
\newblock {\em Electron. J. Probab.}, 19:no. 74, 16, 2014, 1312.5842. \MR{3256874}

\bibitem[Boi90]{indepperc}
D.~Boivin.
\newblock First passage percolation: the stationary case.
\newblock {\em Probability theory and related fields}, 86(4):491--499, 1990.

\bibitem[BP]{bp-lqg-notes}
N.~{Berestycki} and E.~{Powell}.
\newblock {G}aussian {f}ree {f}ield, {L}iouville {q}uantum {g}ravity, and {G}aussian multiplicative chaos.
\newblock {A}vailable at \url{https://homepage.univie.ac.at/nathanael.berestycki/Articles/master.pdf}.

\bibitem[BRG22]{harmonicballs}
A.~Bou-Rabee and E.~Gwynne.
\newblock Harmonic balls in liouville quantum gravity.
\newblock {\em arXiv preprint arXiv:2208.11795}, 2022.

\bibitem[Cer22]{cercle-higher-dimension}
B.~Cercl\'{e}.
\newblock Liouville conformal field theory on even-dimensional spheres.
\newblock {\em J. Math. Phys.}, 63(1):Paper No. 012301, 25, 2022, \arxiv{1912.09219}. \MR{4367620}

\bibitem[CHG24]{appendix}
A.~Contreras~Hip and E.~Gwynne.
\newblock Approximation of length metrics by conformally flat riemannian metrics.
\newblock {\em ArXiv e-prints}, 2024.

\bibitem[DDDF20]{dddf-lfpp}
J.~Ding, J.~Dub\'{e}dat, A.~Dunlap, and H.~Falconet.
\newblock Tightness of {L}iouville first passage percolation for {$\gamma \in (0,2)$}.
\newblock {\em Publ. Math. Inst. Hautes \'{E}tudes Sci.}, 132:353--403, 2020, \arxiv{1904.08021}. \MR{4179836}

\bibitem[DDG21]{ddg-metric-survey}
J.~{Ding}, J.~{Dubedat}, and E.~{Gwynne}.
\newblock {Introduction to the Liouville quantum gravity metric}.
\newblock {\em ArXiv e-prints}, September 2021, \arxiv{2109.01252}.

\bibitem[DF20]{df-lqg-metric}
J.~Dub\'{e}dat and H.~Falconet.
\newblock Liouville metric of star-scale invariant fields: tails and {W}eyl scaling.
\newblock {\em Probab. Theory Related Fields}, 176(1-2):293--352, 2020, \arxiv{1809.02607}. \MR{4055191}

\bibitem[DFG{\etalchar{+}}20]{lqg-metric-estimates}
J.~Dub\'{e}dat, H.~Falconet, E.~Gwynne, J.~Pfeffer, and X.~Sun.
\newblock Weak {LQG} metrics and {L}iouville first passage percolation.
\newblock {\em Probab. Theory Related Fields}, 178(1-2):369--436, 2020, \arxiv{1905.00380}. \MR{4146541}

\bibitem[DG18]{dg-lqg-dim}
J.~{Ding} and E.~{Gwynne}.
\newblock {The fractal dimension of {L}iouville quantum gravity: universality, monotonicity, and bounds}.
\newblock {\em {C}ommunications in {M}athematical {P}hysics}, 374:1877--1934, 2018, \arxiv{1807.01072}.

\bibitem[DG19]{ding-goswami-watabiki}
J.~Ding and S.~Goswami.
\newblock Upper bounds on {L}iouville first-passage percolation and {W}atabiki's prediction.
\newblock {\em Comm. Pure Appl. Math.}, 72(11):2331--2384, 2019, \arxiv{1610.09998}. \MR{4011862}

\bibitem[DGZ23]{dgz-exponential-metric}
J.~{Ding}, E.~{Gwynne}, and Z.~{Zhuang}.
\newblock {Tightness of exponential metrics for log-correlated Gaussian fields in arbitrary dimension}.
\newblock {\em ArXiv e-prints}, October 2023, \arxiv{2310.03996}.

\bibitem[DHKS21]{dhks-even-dim}
L.~{Dello Schiavo}, R.~{Herry}, E.~{Kopfer}, and K.-T. {Sturm}.
\newblock {Conformally invariant random fields, quantum Liouville measures, and random Paneitz operators on Riemannian manifolds of even dimension}.
\newblock {\em ArXiv e-prints}, May 2021, \arxiv{2105.13925}.

\bibitem[DHKS23]{dhks-discrete-to-cont}
L.~{Dello Schiavo}, R.~{Herry}, E.~{Kopfer}, and K.-T. {Sturm}.
\newblock {Polyharmonic Fields and Liouville Quantum Gravity Measures on Tori of Arbitrary Dimension: from Discrete to Continuous}.
\newblock {\em ArXiv e-prints}, February 2023, \arxiv{2302.02963}.

\bibitem[DKRV16]{dkrv-lqg-sphere}
F.~David, A.~Kupiainen, R.~Rhodes, and V.~Vargas.
\newblock Liouville quantum gravity on the {R}iemann sphere.
\newblock {\em Comm. Math. Phys.}, 342(3):869--907, 2016, \arxiv{1410.7318}. \MR{3465434}

\bibitem[DMS21]{wedges}
B.~Duplantier, J.~Miller, and S.~Sheffield.
\newblock Liouville quantum gravity as a mating of trees.
\newblock {\em Ast\'{e}risque}, 427(427):viii+257, 2021, \arxiv{1409.7055}. \MR{4340069}

\bibitem[DRSV17]{lgf-survey}
B.~Duplantier, R.~Rhodes, S.~Sheffield, and V.~Vargas.
\newblock Log-correlated {G}aussian fields: an overview.
\newblock In {\em Geometry, analysis and probability}, volume 310 of {\em Progr. Math.}, pages 191--216. Birkh\"{a}user/Springer, Cham, 2017, \arxiv{1407.5605}. \MR{3821928}

\bibitem[DRV16]{drv-torus}
F.~David, R.~Rhodes, and V.~Vargas.
\newblock Liouville quantum gravity on complex tori.
\newblock {\em J. Math. Phys.}, 57(2):022302, 25, 2016, \arxiv{1504.00625}. \MR{3450564}

\bibitem[DZZ19]{dzz-heat-kernel}
J.~Ding, O.~Zeitouni, and F.~Zhang.
\newblock Heat kernel for {L}iouville {B}rownian motion and {L}iouville graph distance.
\newblock {\em Comm. Math. Phys.}, 371(2):561--618, 2019, \arxiv{1807.00422}. \MR{4019914}

\bibitem[GM20a]{gm-confluence}
E.~Gwynne and J.~Miller.
\newblock Confluence of geodesics in {L}iouville quantum gravity for {$\gamma \in (0,2)$}.
\newblock {\em Ann. Probab.}, 48(4):1861--1901, 2020, \arxiv{1905.00381}. \MR{4124527}

\bibitem[GM20b]{local-metrics}
E.~Gwynne and J.~Miller.
\newblock Local metrics of the {G}aussian free field.
\newblock {\em Ann. Inst. Fourier (Grenoble)}, 70(5):2049--2075, 2020, \arxiv{1905.00379}. \MR{4245606}

\bibitem[GM21a]{uniquenessoflqg}
E.~Gwynne and J.~Miller.
\newblock Existence and uniqueness of the liouville quantum gravity metric for $\gamma \in(0, 2)$.
\newblock {\em Inventiones mathematicae}, 223(1):213--333, 2021.

\bibitem[GM21b]{gm-uniqueness}
E.~Gwynne and J.~Miller.
\newblock Existence and uniqueness of the {L}iouville quantum gravity metric for {$\gamma\in(0,2)$}.
\newblock {\em Invent. Math.}, 223(1):213--333, 2021, \arxiv{1905.00383}. \MR{4199443}

\bibitem[GP19]{gp-kpz}
E.~{Gwynne} and J.~{Pfeffer}.
\newblock {KPZ formulas for the Liouville quantum gravity metric}.
\newblock {\em {T}ransactions of the {A}merican {M}athematical {S}ociety}, to appear, 2019.

\bibitem[GRV19]{grv-higher-genus}
C.~Guillarmou, R.~Rhodes, and V.~Vargas.
\newblock Polyakov's formulation of {$2d$} bosonic string theory.
\newblock {\em Publ. Math. Inst. Hautes \'{E}tudes Sci.}, 130:111--185, 2019, \arxiv{1607.08467}. \MR{4028515}

\bibitem[GS22]{gs-lqg-minkowski}
E.~Gwynne and J.~Sung.
\newblock {T}he {M}inkowski content measure of the {L}iouville quantum gravity metric.
\newblock In preparation, 2022.

\bibitem[HS23]{hs-cardy-embedding}
N.~Holden and X.~Sun.
\newblock Convergence of uniform triangulations under the {C}ardy embedding.
\newblock {\em Acta Math.}, 230(1):93--203, 2023, \arxiv{1905.13207}. \MR{4567714}

\bibitem[JSW18]{imaginarymultchaos}
J.~Junnila, E.~Saksman, and C.~Webb.
\newblock Imaginary multiplicative chaos: Moments, regularity and connections to the ising model. preprint.
\newblock {\em arXiv preprint arXiv:1806.02118}, 2018.

\bibitem[Kah85]{kahane}
J.-P. Kahane.
\newblock Sur le chaos multiplicatif.
\newblock {\em Ann. Sci. Math. Qu\'ebec}, 9(2):105--150, 1985. \MR{829798 (88h:60099a)}

\bibitem[Kin68]{kingman}
J.~F. Kingman.
\newblock The ergodic theory of subadditive stochastic processes.
\newblock {\em Journal of the Royal Statistical Society: Series B (Methodological)}, 30(3):499--510, 1968.

\bibitem[{Le }13]{legall-uniqueness}
J.-F. {Le Gall}.
\newblock Uniqueness and universality of the {B}rownian map.
\newblock {\em Ann. Probab.}, 41(4):2880--2960, 2013, \arxiv{1105.4842}. \MR{3112934}

\bibitem[LG14]{legall-sphere-survey}
J.-F. Le~Gall.
\newblock Random geometry on the sphere.
\newblock In {\em Proceedings of the {I}nternational {C}ongress of {M}athematicians---{S}eoul 2014. {V}ol. 1}, pages 421--442. Kyung Moon Sa, Seoul, 2014, \arxiv{1403.7943}. \MR{3728478}

\bibitem[LSSW16]{fgf-survey}
A.~Lodhia, S.~Sheffield, X.~Sun, and S.~S. Watson.
\newblock Fractional {G}aussian fields: a survey.
\newblock {\em Probab. Surv.}, 13:1--56, 2016, \arxiv{1407.5598}. \MR{3466837}

\bibitem[Mar22]{marzouk-degree-sequence}
C.~Marzouk.
\newblock On scaling limits of random trees and maps with a prescribed degree sequence.
\newblock {\em Ann. H. Lebesgue}, 5:317--386, 2022, \arxiv{1903.06138}. \MR{4443293}

\bibitem[Mie13]{miermont-brownian-map}
G.~Miermont.
\newblock The {B}rownian map is the scaling limit of uniform random plane quadrangulations.
\newblock {\em Acta Math.}, 210(2):319--401, 2013, \arxiv{1104.1606}. \MR{3070569}

\bibitem[MS16]{brownianliouville3}
J.~Miller and S.~Sheffield.
\newblock Liouville quantum gravity and the brownian map iii: the conformal structure is determined.
\newblock {\em arXiv preprint arXiv:1608.05391}, 2016.

\bibitem[MS20a]{brownianliouville1}
J.~Miller and S.~Sheffield.
\newblock Liouville quantum gravity and the brownian map i: the $\mathrm qle(8/3, 0)$ metric.
\newblock {\em Inventiones mathematicae}, 219(1):75--152, 2020.

\bibitem[MS20b]{lqg-tbm1}
J.~Miller and S.~Sheffield.
\newblock Liouville quantum gravity and the {B}rownian map {I}: the {${\mathrm QLE}(8/3,0)$} metric.
\newblock {\em Invent. Math.}, 219(1):75--152, 2020, \arxiv{1507.00719}. \MR{4050102}

\bibitem[MS21a]{lqg-tbm2}
J.~Miller and S.~Sheffield.
\newblock Liouville quantum gravity and the {B}rownian map {II}: {G}eodesics and continuity of the embedding.
\newblock {\em Ann. Probab.}, 49(6):2732--2829, 2021, \arxiv{1605.03563}. \MR{4348679}

\bibitem[MS21b]{brownianliouville2}
J.~Miller and S.~Sheffield.
\newblock Liouville quantum gravity and the brownian map ii: geodesics and continuity of the embedding.
\newblock {\em The Annals of Probability}, 49(6):2732--2829, 2021.

\bibitem[MW17]{miller-wu-dim}
J.~Miller and H.~Wu.
\newblock Intersections of {SLE} {P}aths: the double and cut point dimension of {SLE}.
\newblock {\em Probab. Theory Related Fields}, 167(1-2):45--105, 2017, \arxiv{1303.4725}. \MR{3602842}

\bibitem[Pol81]{polyakov-qg1}
A.~M. Polyakov.
\newblock Quantum geometry of bosonic strings.
\newblock {\em Phys. Lett. B}, 103(3):207--210, 1981. \MR{623209 (84h:81093a)}

\bibitem[Rem18]{remy-annulus}
G.~Remy.
\newblock Liouville quantum gravity on the annulus.
\newblock {\em J. Math. Phys.}, 59(8):082303, 26, 2018, \arxiv{1711.06547}. \MR{3843631}

\bibitem[RV14]{rhodes-vargas-review}
R.~Rhodes and V.~Vargas.
\newblock Gaussian multiplicative chaos and applications: {A} review.
\newblock {\em Probab. Surv.}, 11:315--392, 2014, \arxiv{1305.6221}. \MR{3274356}

\bibitem[She07]{shef-gff}
S.~Sheffield.
\newblock Gaussian free fields for mathematicians.
\newblock {\em Probab. Theory Related Fields}, 139(3-4):521--541, 2007, \arxiv{math/0312099}. \MR{2322706 (2008d:60120)}

\bibitem[WP20]{pw-gff-notes}
W.~{Werner} and E.~{Powell}.
\newblock {Lecture notes on the Gaussian Free Field}.
\newblock {\em ArXiv e-prints}, April 2020, \arxiv{2004.04720}.

\end{thebibliography}
\bibliographystyle{hmralphaabbrv}

\end{document}